\documentclass{article}
\usepackage{makeidx}
\usepackage{amssymb}
\usepackage{amsfonts}
\usepackage{graphicx}
\usepackage{amsmath}
\usepackage{here}

\newtheorem{theorem}{Theorem}

\newtheorem{proposition}[theorem]{Proposition}
\newtheorem{remark}[theorem]{Remark}

\newenvironment{proof}[1][Proof]{\textbf{#1.} }{\ \rule{0.5em}{0.5em}}

\begin{document}

\title{A GCV based Arnoldi-Tikhonov regularization method}
\author{P. Novati, M.R. Russo \\
Department of Mathematics\\
University of Padua, Italy}
\maketitle

\begin{abstract}
For the solution of linear discrete ill-posed problems, in this paper we
consider the Arnoldi-Tikhonov method coupled with the Generalized Cross
Validation for the computation of the regularization parameter at each
iteration. We study the convergence behavior of the Arnoldi method and its
properties for the approximation of the (generalized) singular values, under
the hypothesis that Picard condition is satisfied. Numerical experiments on
classical test problems and on image restoration are presented.
\end{abstract}

\textbf{Key words}. Linear discrete ill-posed problem. Tikhonov
regularization. Arnoldi algorithm. Generalized Cross Validation.

\section{Introduction}

In this paper we consider discrete ill-posed problems,%
\begin{equation}
Ax=b,\quad A\in \mathbb{R}^{N\times N},\text{ }b\in \mathbb{R}^{N},
\label{AvSys}
\end{equation}%
in which the right-hand side $b$ is assumed to be affected by noise, caused
by measurement or discretization errors. These systems typically arise from
the discretization of linear ill-posed problem, such as Fredholm integral
equations of the first kind with compact kernel (see e.g. \cite[Chapter 1]{PCH} for a
background). A common property of these kind of problems, is that the
singular values of the kernel rapidly decay and cluster near zero. In this
situation, provided that the discretization which leads to (\ref{AvSys}) is
consistent with the continuous problem, this property is inherited by the
matrix $A$.

Because of the ill conditioning of $A$ and the presence of noise in $b$,
some sort of regularization is generally employed for solving this kind of
problems. In this framework, a popular and well established regularization
technique is the Tikhonov method, which consists in solving the minimization
problem
\begin{equation}
\min_{x\in \mathbb{R}^{N}}\left\{ \Vert Ax-b\Vert ^{2}+\lambda ^{2}\Vert
Lx\Vert ^{2}\right\} ,  \label{tik}
\end{equation}%
where $\lambda >0$ is the regularization parameter and $L\in \mathbb{R}%
^{P\times N}$ is the regularization matrix (see e.g. \cite{HankHans} and
\cite{PCH} for a background). We denote the solution of (\ref{tik}) by $%
x_{\lambda }$. For a discussion about the choice of $L$ we may quote here
the recent work \cite{Don} and the references therein. As well known, the
choice of the parameter $\lambda $ is crucial in this setting, since it
defines the amount of regularization one wants to impose. Many techniques
have been developed to determine a suitable value for the regularizing
parameter and we can refer to the recent papers \cite{RR,Lukas,Ham,Kind} for
the state of the art, comparison and discussions. We remark that in (\ref%
{tik}) and throughout the paper, the norm used is always the Euclidean norm.

Assuming that $b=\overline{b}+e$, where $\overline{b}$ represents the
unknown error-free right-hand side, in this paper we assume that no
information is available on the error $e$. In such a situation, the most
popular and established techniques for the definition of $\lambda $ in (\ref%
{tik}), as for instance the L-curve criterion and the Generalized Cross
Validation (GCV), typically requires the computation of the GSVD of the
matrix pair $(A,L)$. Of course this decomposition may represents a serious
computational drawback for large-scale problems, such as the image
deblurring. In order to overcome this problem, Krylov projection methods
such as the ones based on the Lanczos bidiagonalization \cite{bjor,Hank,KHE,KO}
and the Arnoldi algorithm \cite{Atfirst,LR} are generally used. Pure
iterative methods such as the GMRES or the LSQR, eventually implemented in a
hybrid fashion (\cite[\S\ 6.6]{PCH}) can also be considered in this
framework.

In this paper we analyze the Arnoldi method for the solution of (\ref{tik})
(the so called Arnoldi-Tikhonov method, introduced in \cite{Atfirst}),
coupled with the GCV as parameter choice rule. Similarly to what made in
\cite{Nagy} for the Lanczos bidiagonalization process, we show that the
resulting algorithm can be fruitfully used for large-scale regularization.
Being based on the orthogonal projection of the matrix $A$ onto the Krylov
subspaces $\mathcal{K}_{m}(A,b)=\mathrm{span}\{b,Ab,\dots ,A^{m-1}b\}$, we
shall observe that for discrete ill-posed problems, the Arnoldi algorithm is
particularly efficient for the approximation of the GCV curve, after a very
few number of iterations.

Indeed, under the hypothesis that Picard condition is satisfied \cite{HanP},
we provide some theoretical results about the convergence of the
Arnoldi-Tikhonov methods and its properties for the approximation of the
singular values of$\ A$. These properties allow us to consider approximation
of the GCV curve which can be obtained working in small dimension (similarly
to what made in \cite{Atfirst} \ where a "projected" L-curve criterion is
used). The GCV curve approximation leads to the definition of a sequence of
regularization parameters (one for each step of the algorithm), which are
fairly good approximation of the regularization parameter arising from the
exact SVD (or GSVD).

The paper is organized as follows. In Section 2 we present a brief outline
about the Arnoldi-Tikhonov method for the iterative solution of (\ref{tik}).
In Section 3 and 4 we provide some theoretical results concerning the
convergence of the Arnoldi algorithm and the SVD (GSVD) approximation. In
Section 5 we explain the use the AT method with the GCV criterion. Some
numerical experiments are presented in Section 6 and 7.

\section{The Arnoldi-Tikhonov method}

Denoting by $\mathcal{K}_{m}(A,b)=\mathrm{span}\{b,Ab,\dots ,A^{m-1}b\}$ the
Krylov subspaces generated by $A$ and the vector $b$, the Arnoldi algorithm
computes an orthonormal basis $\left\{ w_{1},...,w_{m}\right\} $ of $%
\mathcal{K}_{m}(A,b)$. Setting $W_{m}=[w_{1},...,w_{m}]\in \mathbb{R}%
^{N\times m}$, the algorithm can be written in matrix form as%
\begin{equation}
AW_{m}=W_{m}H_{m}+h_{m+1,m}w_{m+1}e_{m}^{T},  \label{dec}
\end{equation}%
where $H_{m}=(h_{i,j})\in \mathbb{R}^{m\times m}$ is an upper Hessenberg
matrix which represents the orthogonal projection of $A$ onto $\mathcal{K}%
_{m}(A,b)$, and $e_{m}=(0,...,0,1)^{T}\in \mathbb{R}^{m}$. Equivalently, the
relation (\ref{dec}) can be written as%
\begin{equation}
AW_{m}=W_{m+1}\overline{H}_{m},  \label{dec2}
\end{equation}%
where%
\begin{equation}
\overline{H}_{m}=\left[
\begin{array}{c}
H_{m} \\
h_{m+1,m}e_{m}^{T}%
\end{array}%
\right] \in \mathbb{R}^{(m+1)\times m}.  \label{hm}
\end{equation}%
In exact arithmetics the Arnoldi process terminates whenever $h_{m+1,m}=0$,
which means that $\mathcal{K}_{m+1}(A,b)=\mathcal{K}_{m}(A,b)$.

If we consider the constrained minimization%
\begin{equation}
\min_{x\in \mathcal{K}_{m}(A,b)}\left\{ \Vert Ax-b\Vert ^{2}+\lambda
^{2}\Vert Lx\Vert ^{2}\right\} ,  \label{mink}
\end{equation}%
writing $x=W_{m}y_{m}$, $y_{m}\in \mathbb{R}^{m}$, and using (\ref{dec2}),
we obtain%
\begin{equation}
\min_{y_{m}\in \mathbb{R}^{m}}\left\{ \left\Vert \overline{H}%
_{m}y_{m}-\left\Vert b\right\Vert e_{1}\right\Vert ^{2}+\lambda
^{2}\left\Vert LW_{m}y_{m}\right\Vert ^{2}\right\} ,  \label{AT}
\end{equation}%
which is known as the Arnoldi-Tikhonov (AT) method. Dealing with Krylov type
solvers, one generally hopes that a good approximation of the exact solution
can be achieved for $m\ll N$, which, in other words, means that the spectral
properties of the matrix $A$ are rapidly simulated by the ones of $\overline{%
H}_{m}$. This method has been introduced in \cite{Atfirst} in the case of $%
L=I_{N}$ (where $I_{N}$ is the identity matrix of order $N$, so that
$\left\Vert LW_{m}y_{m}\right\Vert =\left\Vert
y_{m}\right\Vert $) with the basic aim of reducing the dimension of the
original problem and to avoid the matrix-vector multiplication with $A^{T}$
used by Lanczos type schemes (see \cite{bjor,Hank} and the references
therein).

It is worth noting that (\ref{AT}) can also be interpreted as an hybrid
method. Indeed, the minimization (\ref{AT}) with $L=I_{N}$ is equivalent to
the inner regularization of the GMRES \cite{KO}. We remark however, that for
$L\neq I_{N}$, the philosophy is completely different, since (\ref{AT})
represents the projection of a regularization, while the hybrid approach
aims to regularize the projected problem. As we shall see, this difference
can be appreciated more clearly whenever a parameter choice rule for $%
\lambda $ is adopted.

As well known, in many applications the use of a suitable regularization
operator $L\neq I_{N}$, may substantially improve the quality of the
approximate solution with respect to the choice of $L=I_{N}$. Anyway, we
need to observe that with a general $L\in \mathbb{R}^{P\times N}$, the
minimization (\ref{AT}) is equivalent to%
\begin{equation}
\min_{y_{m}\in \mathbb{R}^{m}}\left\Vert \left(
\begin{array}{c}
\overline{H}_{m} \\
\lambda LW_{m}%
\end{array}%
\right) y_{m}-\left(
\begin{array}{c}
\left\Vert b\right\Vert e_{1} \\
0%
\end{array}%
\right) \right\Vert ^{2},  \label{rp}
\end{equation}%
so that, for $P\approx N$, the dimension of (\ref{rp}) inherits the
dimension of the original problem. Computationally, the situation can be
efficiently faced by means of the "skinny" QR factorization. Anyway,
assuming that $P\leq N$, in order to work with reduced dimension problems,
we add $N-P$ zero rows to $L$ (which does not alter (\ref{mink})) and
consider the orthogonal projection of $L$ onto $\mathcal{K}_{m}(A,b)$, that
is,%
\begin{equation}
L_{m}:=W_{m}^{T}LW_{m}\in \mathbb{R}^{m\times m}.  \label{nlm}
\end{equation}%
This modification leads to the reduced minimization%
\begin{eqnarray}
&&\min_{y_{m}\in \mathbb{R}^{m}}\left\{ \left\Vert \overline{H}%
_{m}y_{m}-\left\Vert b\right\Vert e_{1}\right\Vert ^{2}+\lambda
^{2}\left\Vert L_{m}y_{m}\right\Vert ^{2}\right\}  \label{rm} \\
&=&\min_{x\in \mathcal{K}_{m}(A,b)}\left\{ \Vert Ax-b\Vert ^{2}+\lambda
^{2}\Vert W_{m}^{T}Lx\Vert ^{2}\right\},  \notag
\end{eqnarray}%
which is not equivalent to (\ref{mink}) anymore. Anyway, the use of $L_{m}$
appears natural in this framework, and it is also justified by the fact that%
\begin{equation*}
\left\Vert W_{m}^{T}Lx\right\Vert \leq \Vert Lx\Vert,
\end{equation*}%
since $\left\Vert W_{m}^{T}Lx\right\Vert =\left\Vert
W_{m}W_{m}^{T}Lx\right\Vert $ and $\left\Vert W_{m}W_{m}^{T}\right\Vert =1$,
being $W_{m}W_{m}^{T}$ an orthogonal projection. We observe moreover that $%
L_{m}$ would be the regularization operator of the projection of\ a Franklin
type regularization \cite{Frank}%
\begin{equation*}
\left( A+\lambda L\right) x=b.
\end{equation*}

\section{Convergence analysis for discrete ill-posed problems}

In what follows we denote by $A=U\Sigma V^{T}\in \mathbb{R}^{N\times N}$ the
SVD of $A$ where $\Sigma =diag(\sigma _{1},...,\sigma _{N})$, and by $%
A_{m}:= $ $U_{m}\Sigma _{m}V_{m}^{T}$ the truncated SVD. We remember that
the matrix $\Delta _{m}:=A-A_{m}$ is such that $\left\Vert \Delta
_{m}\right\Vert =\sigma _{m+1}$.

An important property of the methods based on orthogonal projections such as
the Arnoldi algorithm, is the fast theoretical convergence ($%
h_{m+1,m}\rightarrow 0$) if the matrix $A$ comes from the discretization of
operators whose spectrum is clustered around zero. Denote by $\lambda _{j}$,
$j\geq 1$ the eigenvalues of $A$ and assume that $\left\vert \lambda
_{j}\right\vert \geq \left\vert \lambda _{j+1}\right\vert $ for $j\geq 1$.
We have the following result (cf. \cite[Theorem 5.8.10]{Ne}), in which we
assume $N$ arbitrarily large.

\begin{theorem}
\label{nev}Assume that $1\notin \sigma (A)$ and%
\begin{equation}
\sum_{j\geq 1}\sigma _{j}^{p}<\infty \text{ for a certain }0<p\leq 1\text{. }
\label{psum}
\end{equation}%
Let $p_{m}(z)=\prod\nolimits_{i=1}^{m}(z-\lambda _{i})$. Then%
\begin{equation}
\left\Vert p_{m}(A)\right\Vert \leq \left( \frac{\eta e}{m}\right) ^{m/p},
\label{bn}
\end{equation}%
where%
\begin{equation}
\eta (p)\leq \left( 1+p\right) \sum_{j\geq 1}\sigma _{j}^{p}.  \label{etap}
\end{equation}
\end{theorem}

Since%
\begin{equation}
\prod\nolimits_{i=1}^{m}h_{i+1,i}\leq \left\Vert p_{m}(A)b\right\Vert ,
\label{bnp}
\end{equation}%
for each monic polynomial $p_{m}$ of exact degree $m$ (see \cite[p. 269]%
{Trefe}), Theorem \ref{nev} reveals that the rate of decay of $%
\prod\nolimits_{i=1}^{m}h_{i+1,i}$ is superlinear and depends on the $p$%
-summability of the singular values of $A$. We remark that the superlinear
convergence of certain Krylov subspace methods when applied to linear
equations involving compact operators is known in literature (see e.g. \
\cite{Moret} and the references therein). The rate of convergence depends on
the degree of compactness of the operator, which can be measured in terms of
the decay of the singular values.

Here, dealing with severely ill-posed problems, the typical situation is $%
\sigma _{j}=O(e^{-\alpha j})$, where $\alpha>0 $ handles the
degree of ill-conditioning \cite[Definition 2.42]{Hofm}. In this situation,
the following result expresses more clearly the fast decay of $h_{i+1,i}$
with respect to the value of $\alpha $.

\begin{proposition}
\label{pp1}Let $\sigma _{j}=O(e^{-\alpha j})$. Then, for $m\rightarrow
\infty $,%
\begin{equation}
\left( \prod\nolimits_{i=1}^{m}h_{i+1,i}\right) ^{1/m}\leq ke^{-\frac{%
m\alpha }{e^{2}}+\frac{\alpha +2}{2}+O\left( \frac{1}{m}\right) },
\label{ab}
\end{equation}%
where $k$ is a constant independent of $m$.
\end{proposition}

\begin{proof}
Let $k$ be a constant such that $\sigma _{j}\leq ke^{-\alpha j}$. Then for $%
p>0$
\begin{equation}
\eta (p)\leq \left( 1+p\right) \sum_{j\geq 1}\sigma _{j}^{p}\leq k^{p}\frac{%
\left( 1+p\right) }{1-e^{-\alpha p}},  \label{bn1}
\end{equation}%
(cf. (\ref{etap})). Now consider the approximation
\begin{equation*}
k^{p}\frac{\left( 1+p\right) }{1-e^{-\alpha p}}\approx \frac{1}{\alpha p}=:%
\widetilde{\eta }(p),
\end{equation*}%
which is fairly accurate for $p\approx 0$. Using this approximation in (\ref%
{bn}), we find that the minimum of
\begin{equation*}
\left( \frac{\widetilde{\eta }(p)e}{m}\right) ^{m/p},
\end{equation*}%
is attained for $p^{\ast }=\frac{e^{2}}{m\alpha }$. Using this value, the
bound (\ref{bn1}), and defining $t:=\frac{e^{2}}{m}$, we obtain%
\begin{eqnarray*}
\left( \frac{\eta (p^{\ast })e}{m}\right) ^{m/p^{\ast }} &\leq &k^{m}\left(
\frac{\left( 1+p^{\ast }\right) }{1-e^{-\alpha p^{\ast }}}\frac{e}{m}\right)
^{m/p^{\ast }} \\
&=&k^{m}\exp \left( \frac{m\alpha }{t}\ln \left( \frac{1+\frac{t}{\alpha }}{%
1-e^{-t}}\frac{t}{e}\right) \right) \\
&=&k^{m}\exp \left( \frac{m\alpha }{t}\left( -1+t\left( \frac{1}{\alpha }+%
\frac{1}{2}\right) +O(t^{2})\right) \right) \quad \text{for }t\rightarrow 0
\\
&=&k^{m}\exp \left( -\frac{m^{2}\alpha }{e^{2}}+m\left( \frac{\alpha +2}{2}%
\right) +O(1)\right) \quad \text{for }m\rightarrow \infty .
\end{eqnarray*}%
The result immediately follows from (\ref{bnp}) and (\ref{bn}).
\end{proof}

In Figure \ref{F1} (a)-(b) we experimentally
test the bound (\ref{ab}) working with test problems SHAW and WING,
taken from  Hansen's Regularization Toolbox \cite{H1}.
For these two problems it is known that $\alpha
=2 $ and $\alpha =4.5$ respectively.

In the following results we assume to work with problems in which the
discrete Picard condition is satisfied, that is, $u_{m}^{T}b=O(\sigma _{m})$%
, where $u_{m}$ denotes the $m$-th column of $U$, and $b$ is assumed to be
the exact right-hand side.

\begin{proposition}
\label{lem} Assume that the singular values of $A$ are of the type $\sigma
_{j}=O(e^{-\alpha j})$. Assume moreover that the discrete Picard condition
is satisfied. Let $\widetilde{V}_{m}:=\left[ \widetilde{v}_{0},...,%
\widetilde{v}_{m-1}\right] \in \mathbb{R}^{N\times m}$ where $\widetilde{v}%
_{k}:=A^{k}b/\left\Vert A^{k}b\right\Vert $. If $\widetilde{V}_{m}$ has full
column rank, then there exists $C_{m}\in \mathbb{R}^{m\times m}$
nonsingular, $E_{m},F_{m}\in \mathbb{R}^{N\times m}$, such that%
\begin{eqnarray}
\widetilde{V}_{m} &=&U_{m}C_{m}+E_{m},\quad \left\Vert E_{m}\right\Vert =O(%
\sqrt{m}\sigma _{m}),  \label{rr} \\
U_{m} &=&\widetilde{V}_{m}C_{m}^{-1}+F_{m},\quad \left\Vert F_{m}\Sigma
_{m}\right\Vert =O(\sqrt{m}\sigma _{m}).  \label{rr2}
\end{eqnarray}
\end{proposition}

\begin{proof}
Let $U_{m}^{\perp }:=\left[ u_{m+1},...,u_{N}\right] \in \mathbb{R}%
^{(N-m)\times m}$. Writing $c^{(0)}=U_{m}^{T}\widetilde{v}_{0}\in \mathbb{R}%
^{m}$ and $\varepsilon ^{(0)}=U_{m}^{\perp }\left( U_{m}^{\perp }\right) ^{T}%
\widetilde{v}_{0}\in \mathbb{R}^{N}$, we have%
\begin{equation*}
\widetilde{v}_{0}=U_{m}c^{(0)}+\varepsilon ^{(0)}.
\end{equation*}%
The Picard condition implies%
\begin{equation*}
\left\Vert \varepsilon ^{(0)}\right\Vert =\left\Vert \left( U_{m}^{\perp
}\right) ^{T}\widetilde{v}_{0}\right\Vert =O(\sigma _{m}),
\end{equation*}%
since $\sigma _{j}=O(e^{-\alpha j})$ and then using%
\begin{equation}
\left( \sum\limits_{j\geq m+1}e^{-2\alpha j}\right) ^{1/2}\leq \frac{1}{%
\sqrt{2\alpha }}e^{-\alpha m}.  \label{sum}
\end{equation}%
From the relation $\left\Vert A-U_{m}\Sigma _{m}V_{m}^{T}\right\Vert
=\sigma _{m+1}$, after some computation one easily finds that for $0<k\leq
m-1 $,%
\begin{equation*}
\widetilde{v}_{k}=U_{m}c^{(k)}+\varepsilon ^{(k)},
\end{equation*}%
where $c^{(k)}=U_{m}^{T}\widetilde{v}_{k}\in \mathbb{R}^{m}$ and%
\begin{equation*}
\varepsilon ^{(k)}=\frac{\left\Vert A^{k-1}b\right\Vert }{\left\Vert
A^{k}b\right\Vert }A\varepsilon ^{(k-1)}+O(\sigma _{m+1}),
\end{equation*}%
so that $\left\Vert \varepsilon ^{(k)}\right\Vert =O(\sigma _{m})$. Defining
$C_{m}=\left[ c^{(0)},...,c^{(m-1)}\right] =U_{m}^{T}\widetilde{V}_{m}\in
\mathbb{R}^{m\times m}$ and $E_{m}=\left[ \varepsilon ^{(0)},...,\varepsilon
^{(m-1)}\right] \in \mathbb{R}^{N\times m}$ we have proved (\ref{rr}).

By (\ref{rr}), we can write%
\begin{equation}
U_{m}=\widetilde{V}_{m}C_{m}^{-1}-E_{m}C_{m}^{-1},  \label{um}
\end{equation}%
and since $E_{m}=U_{m}^{\perp }\left( U_{m}^{\perp }\right) ^{T}\widetilde{V}%
_{m}$ we have that
\begin{equation}
E_{m}C_{m}^{-1}=U_{m}^{\perp }\left( U_{m}^{\perp }\right) ^{T}\widetilde{V}%
_{m}\left( U_{m}^{T}\widetilde{V}_{m}\right) ^{-1}.  \label{em}
\end{equation}%
Now observe that (\ref{rr}) implies%
\begin{equation*}
\left( U_{m}^{\perp }\right) ^{T}\widetilde{V}_{m}=\left(
\begin{array}{ccc}
O(\sigma _{m+1}) & \cdots & O(\sigma _{m+1}) \\
\vdots &  & \vdots \\
O(\sigma _{N}) & \cdots & O(\sigma _{N})%
\end{array}%
\right) \in \mathbb{R}^{(N-m)\times m},
\end{equation*}%
and%
\begin{equation*}
U_{m}^{T}\widetilde{V}_{m}=\left(
\begin{array}{ccc}
O(\sigma _{1}) & \cdots & O(\sigma _{1}) \\
\vdots &  & \vdots \\
O(\sigma _{m}) & \cdots & O(\sigma _{m})%
\end{array}%
\right) \in \mathbb{R}^{m\times m}.
\end{equation*}%
Using the Cramer rule to invert $U_{m}^{T}\widetilde{V}_{m}$ we find that
each entry of $\left( U_{m}^{T}\widetilde{V}_{m}\right) ^{-1}\Sigma _{m}\in
\mathbb{R}^{m\times m}$ is of the type $O(1)$, and hence%
\begin{equation}
\left( U_{m}^{\perp }\right) ^{T}\widetilde{V}_{m}\left( U_{m}^{T}\widetilde{%
V}_{m}\right) ^{-1}\Sigma _{m}=\left(
\begin{array}{ccc}
O(\sigma _{m+1}) & \cdots & O(\sigma _{m+1}) \\
\vdots &  & \vdots \\
O(\sigma _{N}) & \cdots & O(\sigma _{N})%
\end{array}%
\right) \in \mathbb{R}^{(N-m)\times m}.  \label{mat}
\end{equation}%
Defining $F_{m}=-E_{m}C_{m}^{-1}$ we obtain (\ref{rr2}) by (\ref{um}), (\ref%
{em}) and (\ref{mat}), and applying (\ref{sum}).
\end{proof}

\begin{remark}
The hypothesis $\sigma _{j}=O(e^{-\alpha j})$ of Proposition \ref{lem} is
just used to have $\left\Vert \varepsilon ^{(0)}\right\Vert =O(\sigma _{m})$
by (\ref{sum}). The result of the proposition can be extended to work with
moderately ill-posed problems, in which $\sigma _{j}=O(j^{-\alpha })$,
provided that $\alpha $ is large enough. As consequence in this situation we
would have a slower decay of $\left\Vert E_{m}\right\Vert $ and $\left\Vert
F_{m}\Sigma _{m}\right\Vert $.
\end{remark}

The following result improves the one of Theorem \ref{nev} (which holds
without hypothesis on $b$).

\begin{proposition}
\label{pp2}Under the hypothesis of Proposition \ref{lem}%
\begin{equation*}
h_{m+1,m}=O(\sqrt{m}\sigma _{m}).
\end{equation*}
\end{proposition}

\begin{proof}
By (\ref{dec})%
\begin{eqnarray*}
h_{m+1,m} &=&w_{m+1}^{T}Aw_{m} \\
&=&w_{m+1}^{T}\Delta _{m}w_{m}+w_{m+1}^{T}A_{m}w_{m} \\
&=&O(\sigma _{m+1})+w_{m+1}^{T}U_{m}\Sigma _{m}V_{m}^{T}w_{m},
\end{eqnarray*}%
since $\left\Vert \Delta _{m}\right\Vert =\sigma _{m+1}$. Therefore, using (%
\ref{rr2}) we obtain%
\begin{equation*}
h_{m+1,m}=O(\sigma _{m+1})+w_{m+1}^{T}(\widetilde{V}_{m}C_{m}^{-1}+F_{m})%
\Sigma _{m}V_{m}^{T}w_{m}.
\end{equation*}%
which concludes the proof, since $w_{m+1}^{T}\widetilde{V}_{m}=0$ and $%
\left\Vert F_{m}\Sigma _{m}\right\Vert =O(\sqrt{m}\sigma _{m})$.
\end{proof}

In Figure \ref{F1} (c)-(d) we compare the decay of the sequence $\left\{
h_{m+1,m}\right\} _{m\geq 1}$ with that of the singular values, working
again with the test problems SHAW and WING.

\begin{figure}[H]
\centering
\includegraphics[width=0.45\textwidth]{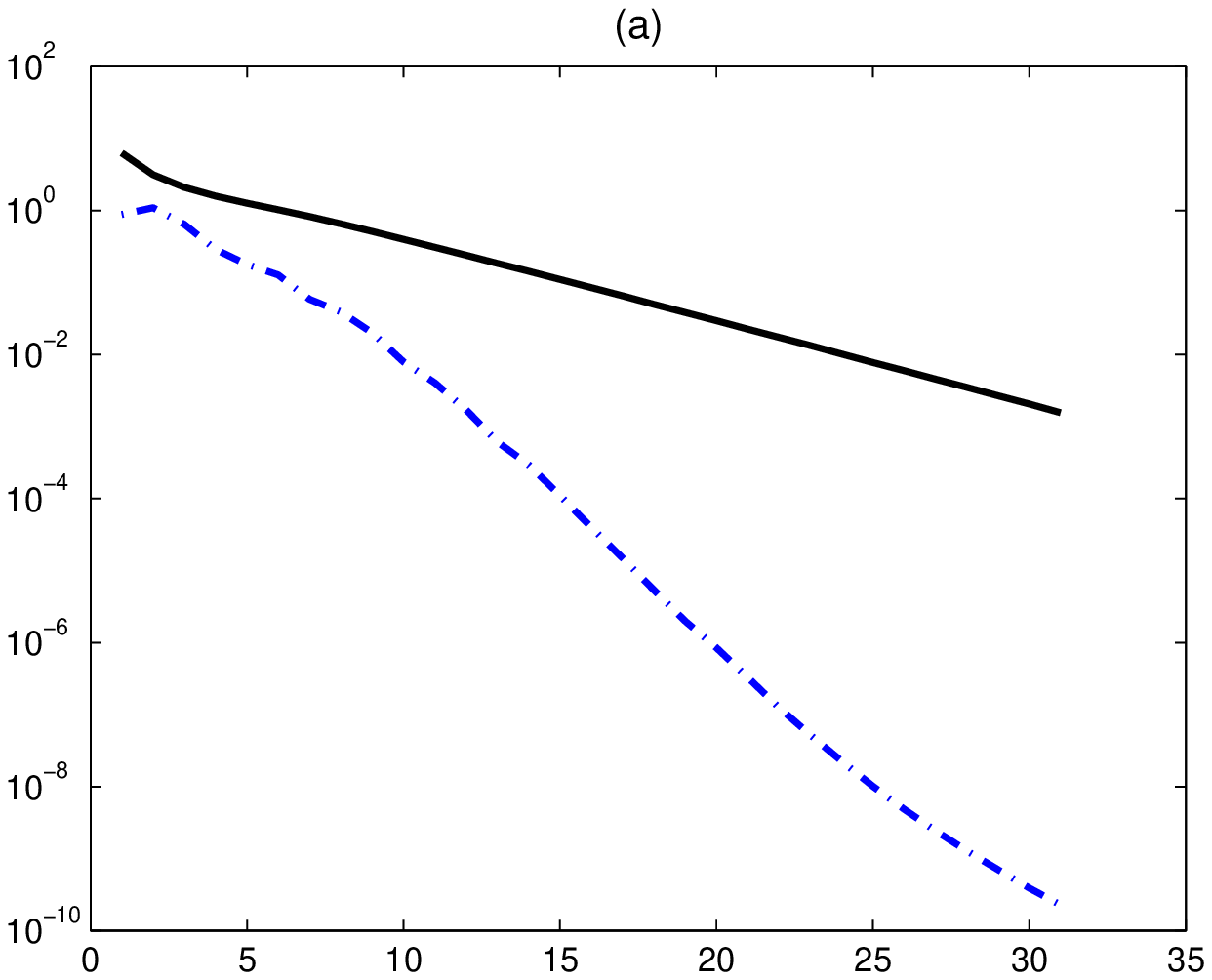} %
\includegraphics[width=0.45\textwidth]{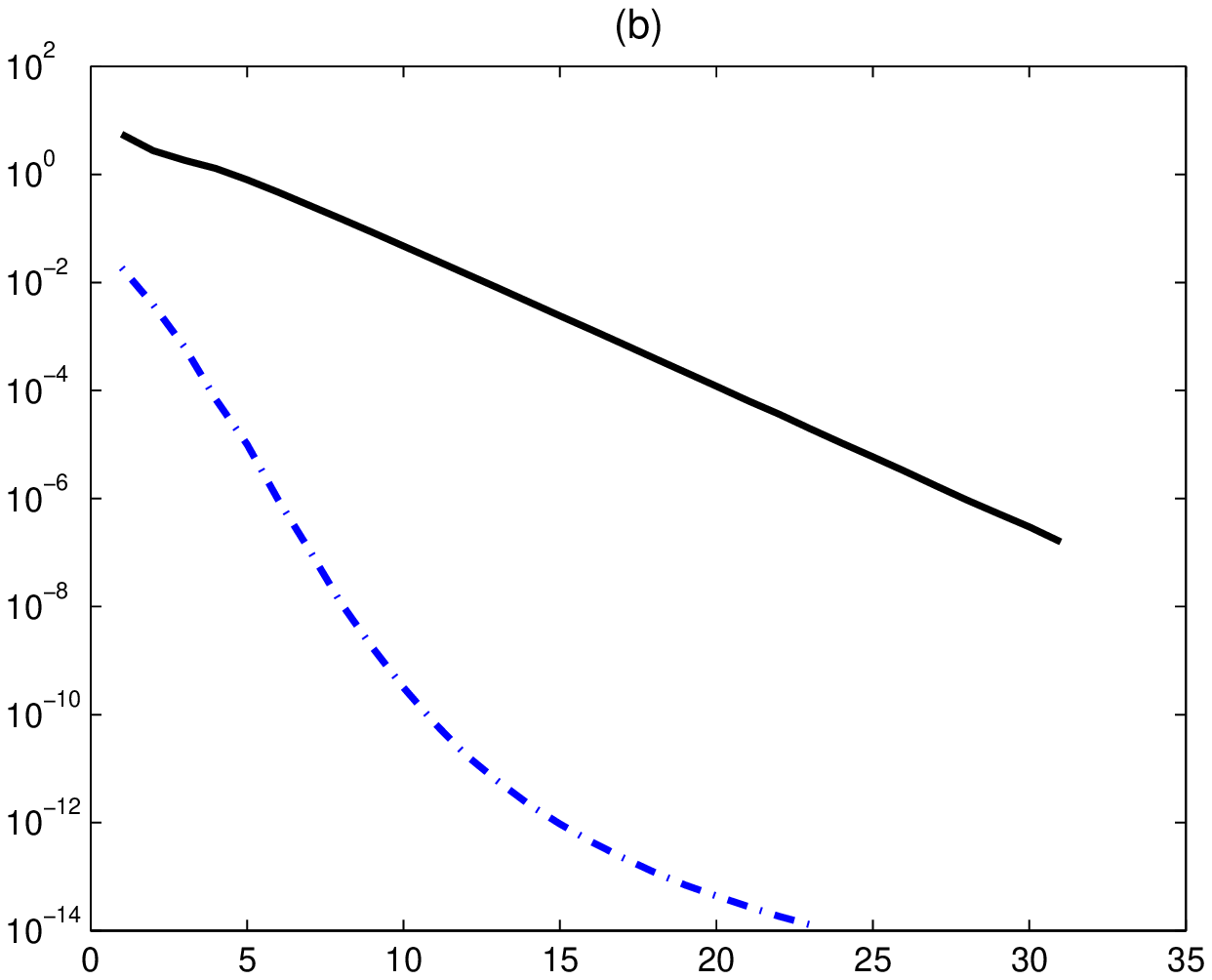}\newline
\includegraphics[width=0.45\textwidth]{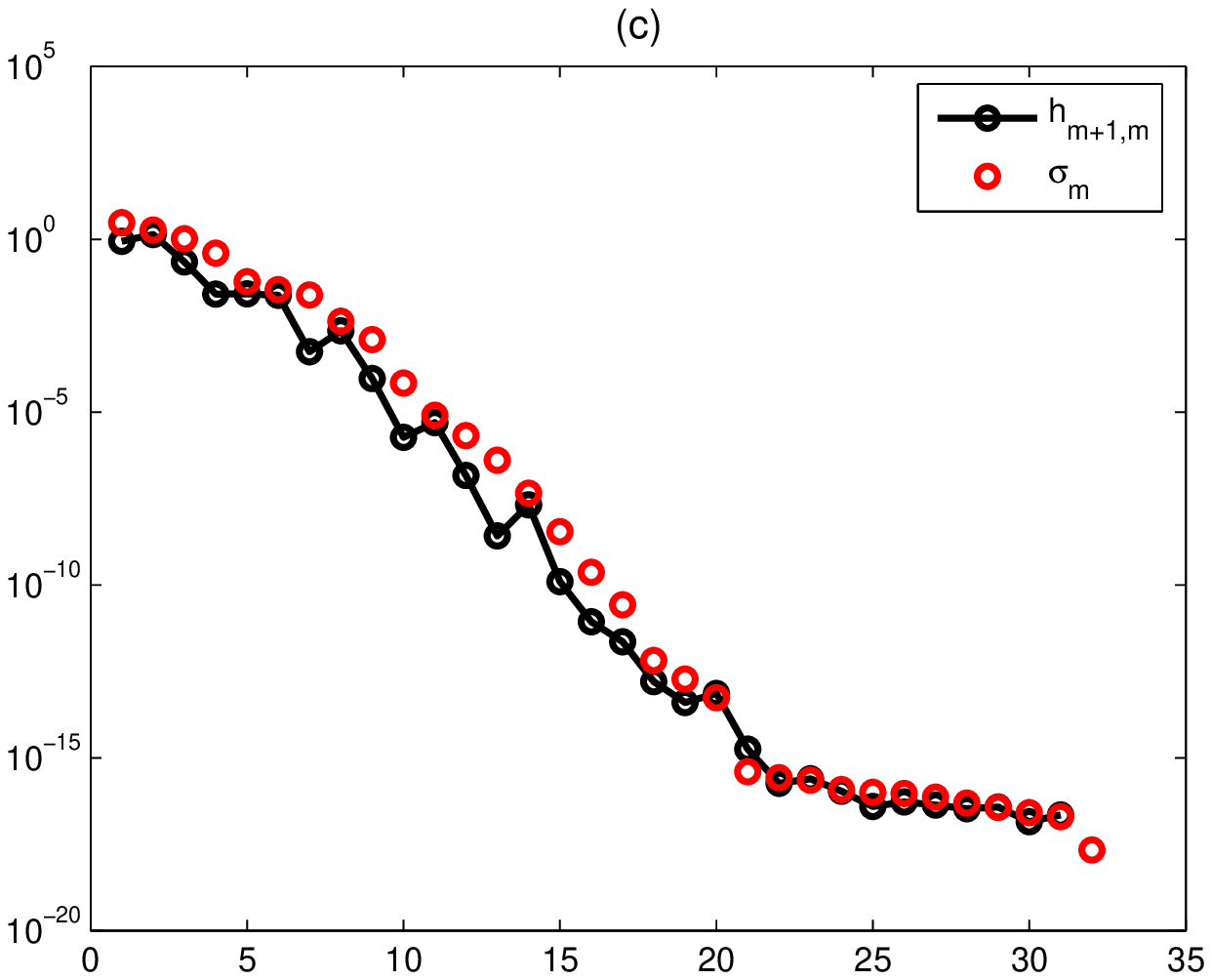} %
\includegraphics[width=0.45\textwidth]{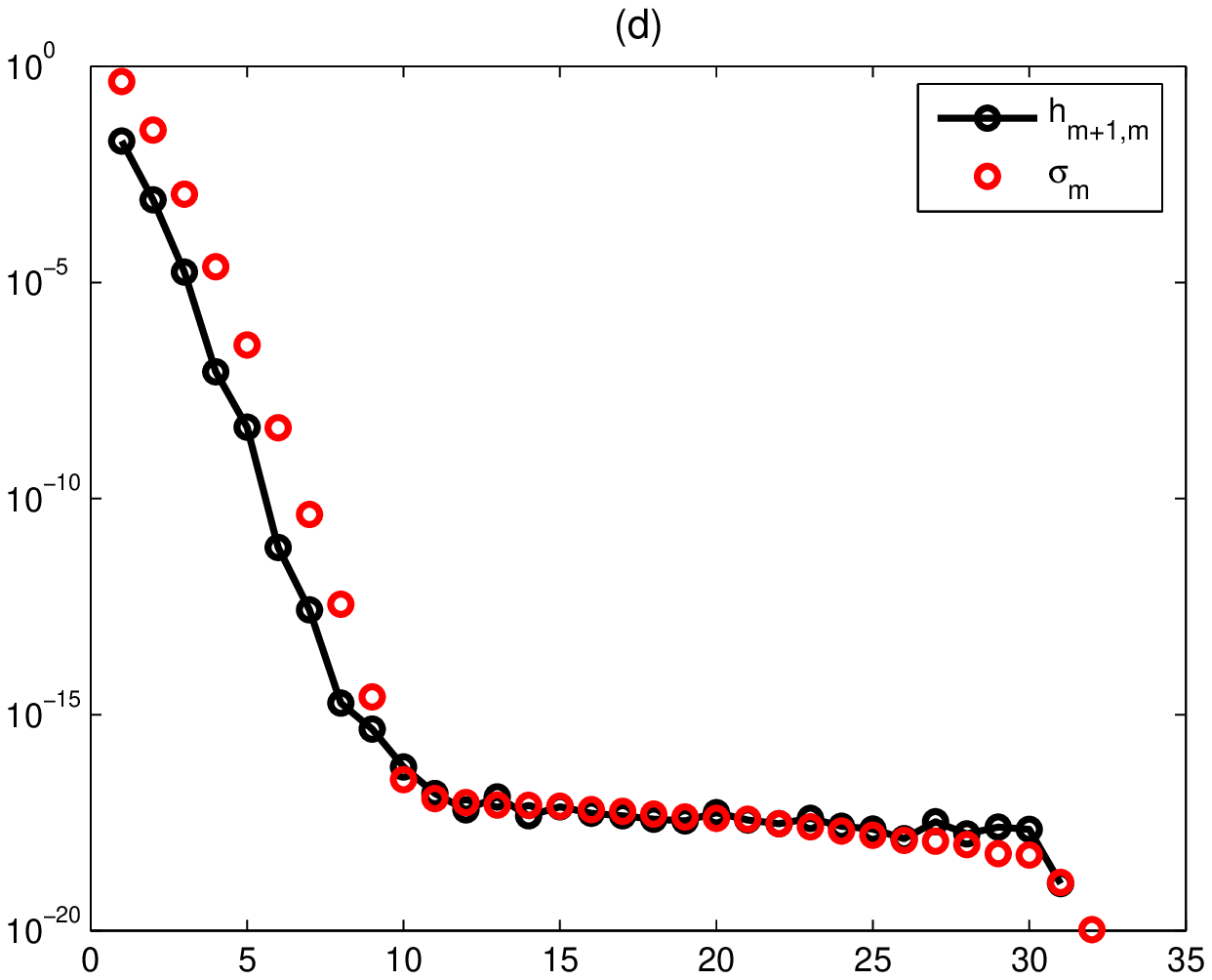}
\caption{\textit{(a)-(b) decay rate of \ $\left(\prod%
\nolimits_{i=1}^{m}h_{i+1,i}\right) ^{1/m}$ (dash-dot line) and bound (\protect\ref{ab}) (solid line),
(c)-(d) decay of $h_{m+1,m}$
and $\protect\sigma _{m}$. On the left the results for SHAW and on the right
the results for WING. In each experiment $N=32$.}}
\label{F1}
\end{figure}

We need to remark that the results of Figure \ref{F1} are obtained working
with the Householder implementation of the Arnoldi algorithm and hence
simulating what happens in exact arithmetics.

\section{The approximation of the SVD}

The use of the Arnoldi algorithm as a method to approximate the marginal
values of the spectrum of a matrix is widely known in literature. We may
refer to \cite[Chapter 6]{Saad} for an exhaustive background. Using similar
arguments, in this section we analyze the convergence of the singular values
of the matrices $\overline{H}_{m}$ to the largest singular values of $A$.
For the Lanczos bidiagonalization method \cite{bjor, OS}, the analysis can
be done by exploiting the connection between this method and the symmetric
Lanczos process (see e.g. \cite{GolLukOv}). The use of the Lanczos
bidiagonalization to construct iteratively the GSVD of ($A,L$) has been
studied in \cite{KHE}.

Let us consider the SVD factorization of $\overline{H}_{m}$, that is, $%
\overline{H}_{m}=U^{(m)}\Sigma ^{(m)}V^{(m)T}$, $U^{(m)}\in \mathbb{R}%
^{\left( m+1\right) \times \left( m+1\right) }$, $V^{(m)}\in \mathbb{R}%
^{m\times m}$ and%
\begin{equation*}
\Sigma ^{(m)}=\left(
\begin{array}{ccc}
\sigma _{1}^{(m)} &  &  \\
& \ddots &  \\
&  & \sigma _{m}^{(m)} \\
0 & \cdots & 0%
\end{array}%
\right) \in \mathbb{R}^{\left( m+1\right) \times m}.
\end{equation*}%
We can state the following results.

\begin{proposition}
Let $\overline{U}_{m+1}=W_{m+1}U^{(m)}\in \mathbb{R}^{N\times \left(
m+1\right) }$ and $\overline{V}_{m}=W_{m}V^{(m)}\in \mathbb{R}^{N\times m}$.
Then%
\begin{equation*}
\left\Vert A-\overline{U}_{m+1}\Sigma ^{(m)}\overline{V}_{m}^{T}\right\Vert
=\left\Vert A(I-W_{m}W_{m}^{T})\right\Vert.
\end{equation*}
\end{proposition}

\begin{proof}
Using (\ref{dec2}), we have%
\begin{eqnarray*}
A-\overline{U}_{m+1}\Sigma ^{(m)}\overline{V}_{m}^{T}
&=&A-W_{m+1}U^{(m)}\Sigma ^{(m)}V^{(m)T}W_{m}^{T} \\
&=&A-W_{m+1}\overline{H}_{m}W_{m}^{T} \\
&=&A-AW_{m}W_{m}^{T}.
\end{eqnarray*}
\end{proof}

Observe that since $\overline{U}_{m+1}\Sigma ^{(m)}=W_{m+1}\widetilde{U}^{(m)}
\widetilde{\Sigma }^{(m)}$, where $\widetilde{\Sigma }^{(m)}\in \mathbb{R}^{m\times m}$
is just $\Sigma ^{(m)}$ without the last row, and  $\widetilde{U}^{(m)} \in \mathbb{R}^{(m+1)\times m}$
is $U^{(m)}$ without the last column,
the above result states that
the triplet $\left( W_{m+1}\widetilde{U}^{(m)},\widetilde{\Sigma }^{(m)},W_{m}V^{(m)}%
\right) $ defines an approximation of the truncated SVD of $A$, which cannot
be too bad since $\left\Vert A(I-W_{m}W_{m}^{T})\right\Vert \leq \left\Vert
A\right\Vert $. Moreover, it states that if the Arnoldi algorithm does not
terminate before $N$ iterations, then it produces the complete SVD. The
following result gives some additional information.

\begin{proposition}
\label{p1}Let $u_{k}^{(m)}\in \mathbb{R}^{m+1}$ and $v_{k}^{(m)}\in \mathbb{R%
}^{m}$ be respectively the right and left singular vectors relative to the
singular value $\sigma _{k}^{(m)}$ of $\overline{H}_{m}$, that is, $%
\overline{H}_{m}v_{k}^{(m)}=\sigma _{k}^{(m)}u_{k}^{(m)}$ and $\overline{H}%
_{m}^{T}u_{k}^{(m)}=\sigma _{k}^{(m)}v_{k}^{(m)}$, with $1\leq k\leq m$.
Then defining $\overline{u}_{k}=W_{m+1}u_{k}^{(m)}$ and $\overline{v}%
_{k}=W_{m}v_{k}^{(m)}$ we have that%
\begin{eqnarray}
A\overline{v}_{k}-\sigma _{k}^{(m)}\overline{u}_{k} &=&0,  \label{r1} \\
W_{m}^{T}(A^{T}\overline{u}_{k}-\sigma _{k}^{(m)}\overline{v}_{k}) &=&0.
\label{r2}
\end{eqnarray}
\end{proposition}

\begin{proof}
(\ref{r1}) follows directly by (\ref{dec2}). Moreover, since%
\begin{equation*}
\overline{H}_{m}^{T}u_{k}^{(m)}-\sigma _{k}^{(m)}v_{k}^{(m)}=0,
\end{equation*}%
using $\overline{H}_{m}^{T}=W_{m}^{T}A^{T}W_{m+1}$, and the definition of $%
\overline{u}_{k}$ and $\overline{v}_{k}$, we easily obtain (\ref{r2}).
\end{proof}

\begin{remark}
Using the square matrix $H_{m}$ to approximate the singular values of $A$,
that is, computing the SVD $H_{m}=U^{(m)}\Sigma ^{(m)}V^{(m)T}$, where now $%
U^{(m)},\Sigma ^{(m)},V^{(m)}\in \mathbb{R}^{m\times m}$, if $%
H_{m}v_{k}^{(m)}=\sigma _{k}^{(m)}u_{k}^{(m)}$ then
\begin{equation}
\left\Vert A\overline{v}_{k}-\sigma _{k}^{(m)}\overline{u}_{k}\right\Vert
\leq h_{m+1,m}\quad \text{with\quad }\overline{u}_{k}=W_{m}u_{k}^{(m)},%
\overline{v}_{k}=W_{m}v_{k}^{(m)}.  \label{rs}
\end{equation}%
The above relation is very similar to the one which arises when using the
eigenvalues of $H_{m}$ (the Ritz values) to approximate the eigenvalues of $%
A $ \cite[\S 6.2]{Saad}. Note moreover that whenever $h_{m+1,m}\approx 0$,
and hence very quickly for linear ill-posed problems (see Section 3), the
use of $\overline{H}_{m}$ or $H_{m}$ is almost equivalent to approximate the
largest singular values of $A$.
\end{remark}

The Galerkin condition (\ref{r2}) is consequence of the fact that the
Arnoldi algorithm does not work with the transpose. Obviously, if $A=A^{T}$,
the algorithm reduces to the symmetric Lanczos process and, under the
hypothesis of Proposition \ref{p1}, we easily obtain $A^{T}\overline{u}%
_{k}-\sigma _{k}^{(m)}\overline{v}_{k}=0$. In the general case of $A\neq
A^{T}$, Proposition \ref{p1} ensures that since $\overline{v}%
_{k}=W_{m}v_{k}^{(m)}\in \mathcal{K}_{m}(A,b)$, by (\ref{r2}) the vector $%
\sigma _{k}^{(m)}\overline{v}_{k}$ is just the orthogonal projection of $%
A^{T}\overline{u}_{k}$ onto $\mathcal{K}_{m}(A,b)$, that is, $\sigma
_{k}^{(m)}\overline{v}_{k}=W_{m}W_{m}^{T}A^{T}\overline{u}_{k}$, which
implies%
\begin{equation}
\left\Vert A^{T}\overline{u}_{k}-\sigma _{k}^{(m)}\overline{v}%
_{k}\right\Vert \leq \left\Vert
(I-W_{m}W_{m}^{T})A^{T}W_{m}W_{m}^{T}\right\Vert .  \label{r3}
\end{equation}%
This means that the approximation is good if $A^{T}\overline{u}_{k}$ is
close to $\mathcal{K}_{m}(A,b)$. It is interesting to observe that (\ref{r3})
is just the "transpose version" of (\ref{rs}) since%
\begin{equation*}
h_{m+1,m}=\left\Vert (I-W_{m}W_{m}^{T})AW_{m}W_{m}^{T}\right\Vert,
\end{equation*}%
which can be easily proved using again (\ref{dec}) (cf. \cite[Chapter 4]%
{Saad}).

Experimentally, one observes that the Arnoldi algorithm seems to be very
efficient for approximating the largest singular values for discrete
ill-posed problems. In order to have a-posteriori strategy to monitor
step-by-step the quality of approximation, we can state the following.

\begin{proposition}
\label{pp}Assume that the matrix $A$ has full rank. Then
\begin{equation}
\left\Vert A^{T}\overline{u}_{k}-\sigma _{k}^{(m)}\overline{v}%
_{k}\right\Vert \leq \left\Vert W_{m+1}^{T}AW_{m}^{\bot }\right\Vert ,
\label{ddm}
\end{equation}%
where $\overline{u}_{k}$, $\overline{v}_{k}$, $\sigma _{k}^{(m)}$ are
defined as in Proposition \ref{p1}, and $W_{m}^{\bot }=[w_{m+1},...,w_{N}]$.
\end{proposition}

\begin{proof}
Since $\overline{v}_{k}\in \mathcal{K}_{m}(A,b)$, and $\overline{u}%
_{k}=W_{m+1}u_{k}^{(m)}$, by (\ref{r2})
\begin{equation}
\left\Vert A^{T}\overline{u}_{k}-\sigma _{k}^{(m)}\overline{v}%
_{k}\right\Vert \leq \left\Vert \left( W_{m}^{\bot }\right)
^{T}A^{T}W_{m+1}\right\Vert .  \label{rr3}
\end{equation}
\end{proof}

Formula (\ref{ddm}) is rather interesting because since $%
h_{ij}=w_{i}^{T}Aw_{j}$ from the Arnoldi algorithm,%
\begin{equation*}
W_{m+1}^{T}AW_{m}^{\bot }=\left[
\begin{array}{ccc}
h_{1,m+1} & \cdots & h_{1,N} \\
\vdots &  & \vdots \\
h_{m+1,m+1} & \cdots & h_{m+1,N}%
\end{array}%
\right] .
\end{equation*}%
Since in many cases the elements of the projected matrix $H_{m}$ tends to
annihilates departing from the diagonal (this is the basic assumption of the
methods based on the incomplete orthogonalization, see e.g. \cite{Saad2}),
one may obtain useful estimates for the bound (\ref{ddm}) working with few
columns of $W_{m+1}^{T}AW_{m}^{\bot }$, that is, with few columns of $%
W_{m}^{\bot }$, and hence obtaining a-posteriori estimates for the quality of
the SVD approximation. In order to have an experimental confirmation of this
statement, in Figure \ref{F2} we show the behavior of $\left\Vert A-%
\overline{U}_{m+1}\Sigma ^{(m)}\overline{V}_{m}^{T}\right\Vert $ and $%
\left\Vert W_{m+1}^{T}Aw_{m+1}\right\Vert $, for some test problems.
Note that $\left\Vert W_{m+1}^{T}Aw_{m+1}\right\Vert $ comes from the bound (\ref{ddm})
with $W_{m}^{\bot }$ replaced by $w_{m+1}.$

We remark that Proposition \ref{pp2} and \ref{pp} can be used to arrest the procedure whenever
the noise level $\varepsilon$ is known, since it is generally useless to continue with the SVD approximation if we find $\sigma _{k}^{(m)} << \varepsilon $, for a certain $k$ and $m$. Indeed,
in this situation the Picard condition is no longer satisfied since typically
$U_{m}^{T}\approx \varepsilon$ for $m$ large enough.

For what concerns the generalized SVD of the matrix pair $(A,L)$, let $AX=US$
and $LX=VC$, where $S=diag(s_{1},...,s_{N})$ and $C=diag(c_{1},...,c_{N})$, $%
X\in \mathbb{R}^{N\times N}$ is nonsingular and $U,V\in \mathbb{R}^{N\times
N}$ are orthogonal. Moreover let $\overline{H}_{m}X^{(m)}=U^{(m)}S^{(m)}$\
and $L_{m}X^{(m)}=V^{(m)}C^{(m)}$, where $%
S^{(m)}=diag(s_{1}^{(m)},...,s_{m}^{(m)})$ and $%
C^{(m)}=diag(c_{1}^{(m)},...,c_{m}^{(m)})$, be the generalized SVD of the
matrix pair $(\overline{H}_{m},L_{m})$. In this situation, for the
convergence of the approximated generalized singular values and vectors, we
can state the following result.

\begin{figure}[H]
\centering
\includegraphics[width=0.45\textwidth]{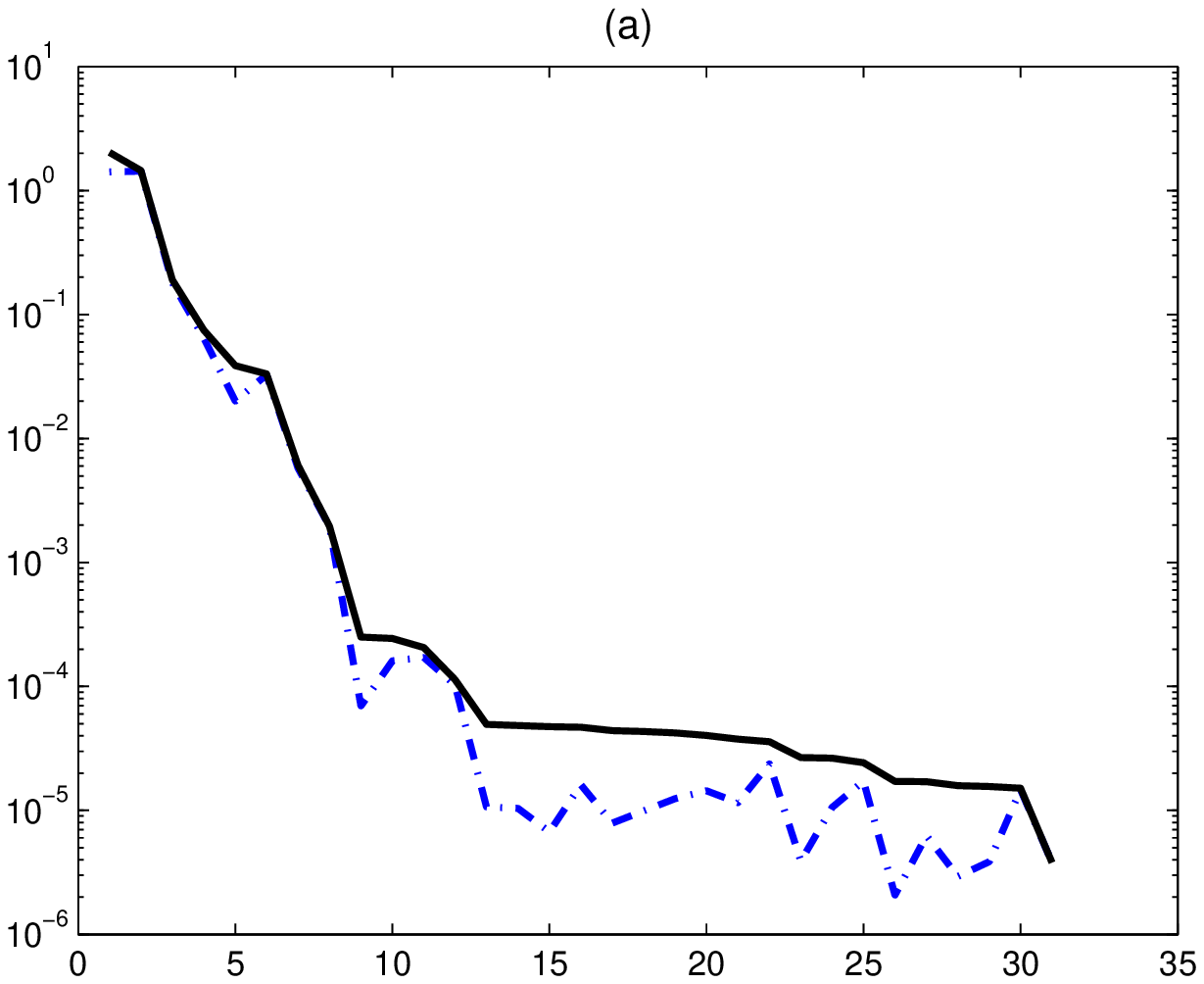} %
\includegraphics[width=0.45\textwidth]{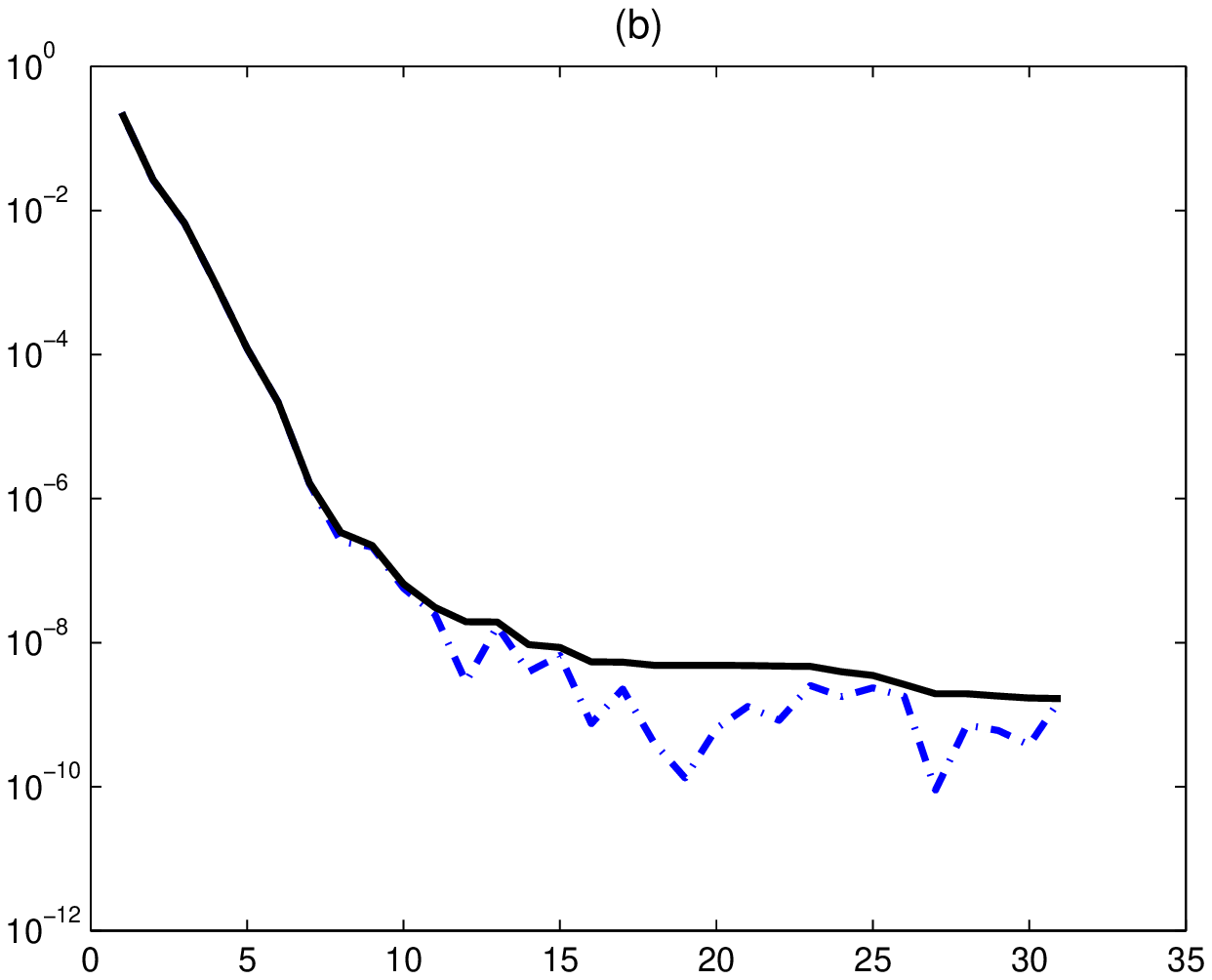}\newline
\includegraphics[width=0.45\textwidth]{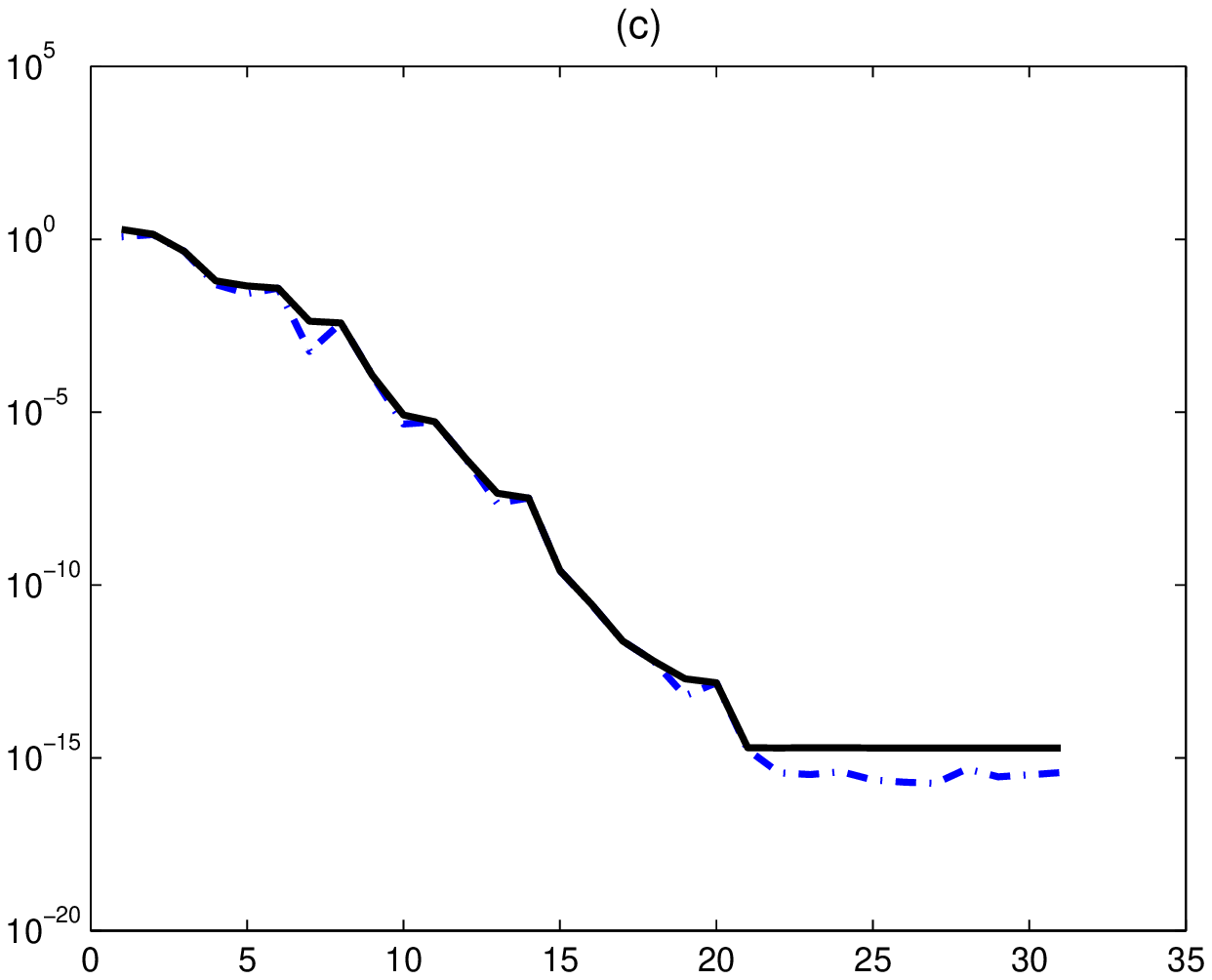} %
\includegraphics[width=0.45\textwidth]{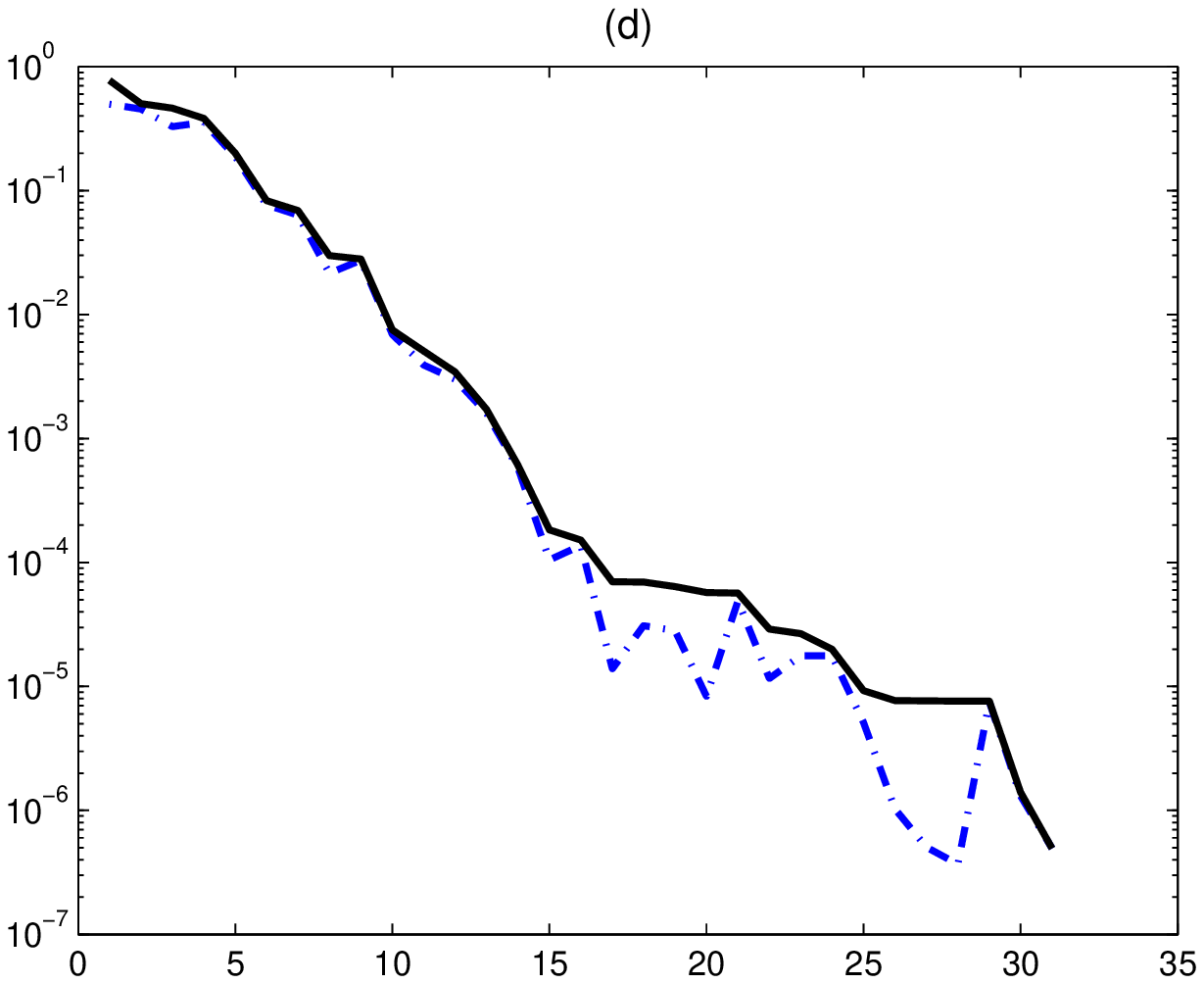}
\caption{\textit{Decay behavior of }$\left\Vert A-\overline{U}_{m+1}\Sigma
^{(m)}\overline{V}_{m}^{T}\right\Vert $\textit{\ (solid line) and lower bound} $\left\Vert
W_{m+1}^{T}Aw_{m+1}\right\Vert $\textit{\ arising from Proposition \ref{pp} (dash-dot line) for BAART (a),
WING (b), SHAW (c) and I\_LAPLACE (d). The dimension of each problem is }$N=32$\textit{.}
}
\label{F2}
\end{figure}

\begin{proposition}
\label{p2}Let $u_{k}^{(m)}$, $v_{k}^{(m)}$ and $x_{k}^{(m)}$ be the $k$-th
column of the matrices $U^{(m)}\in \mathbb{R}^{\left( m+1\right) \times m}$,
$V^{(m)}\in \mathbb{R}^{m\times m}$ and $X^{(m)}\in \mathbb{R}^{m\times m}$
respectively. Then defining $\overline{u}_{k}=W_{m+1}u_{k}^{(m)}$, $%
\overline{v}_{k}=W_{m}v_{k}^{(m)}$ and $\overline{x}_{k}=W_{m}x_{k}^{(m)}$,
we have%
\begin{eqnarray}
A\overline{x}_{k}-s_{k}^{(m)}\overline{u}_{k} &=&0,  \label{g1} \\
W_{m}^{T}(L\overline{x}_{k}-c_{k}^{(m)}\overline{v}_{k}) &=&0.  \label{g2}
\end{eqnarray}
\end{proposition}

\begin{proof}
Similarly to Proposition \ref{p1}, (\ref{g1}) and (\ref{g2}) follows
immediately from the basic relation (\ref{dec2}).
\end{proof}

As before the proposition ensures that if the matrix $A$ has full rank, than
the Arnoldi algorithm allows to construct the GSVD of $(A,L)$. Step by step,
the quality of the approximation depends on the distance between $%
span\{Lw_{1},...,Lw_{m}\}$ and $\mathcal{K}_{m}(A,b)$. Similarly to (\ref{r3}%
) and (\ref{rr3}), since $\overline{v}_{k}=W_{m}v_{k}^{(m)}\in \mathcal{K}%
_{m}(A,b)$, we have%
\begin{equation*}
\left\Vert L\overline{x}_{k}-c_{k}^{(m)}\overline{v}_{k}\right\Vert \leq
\left\Vert (I-W_{m}W_{m}^{T})LW_{m}W_{m}^{T}\right\Vert .
\end{equation*}%
and%
\begin{equation*}
\left\Vert L\overline{x}_{k}-\sigma _{k}^{(m)}\overline{v}_{k}\right\Vert
\leq \left\Vert \left( W_{m}^{\bot }\right) ^{T}LW_{m}\right\Vert .
\end{equation*}

In Figure \ref{F3} we show the convergence of the singular values of $%
\overline{H}_{m}$, and the generalized singular values of the matrix pair $%
\left( \overline{H}_{m},L_{m}\right) $, with
\begin{equation*}
L=\left(
\begin{array}{cccc}
1 & -1 &  &  \\
& \ddots  & \ddots  &  \\
&  & 1 & -1 \\
0 & \cdots  & \cdots  & 0%
\end{array}%
\right) ,
\end{equation*}%
working with the test problems SHAW and BAART. The results show that the
approximations are quite accurate. It is interesting to observe that, in
both cases, after 8-9 iterations the algorithm starts to generate spurious
approximations. This is due to the loss of orthogonality of the Krylov
vectors, since in these experiments (and in what follows) we have used the
Gram-Schmidt implementation. Working with the Householder version of the
algorithm the problem is fixed. Anyway in the framework of the
regularization, a more accurate approximation of the smallest singular
values is useless because of the error in $b$.

\begin{figure}[H]
\centering
\includegraphics[width=0.45\textwidth]{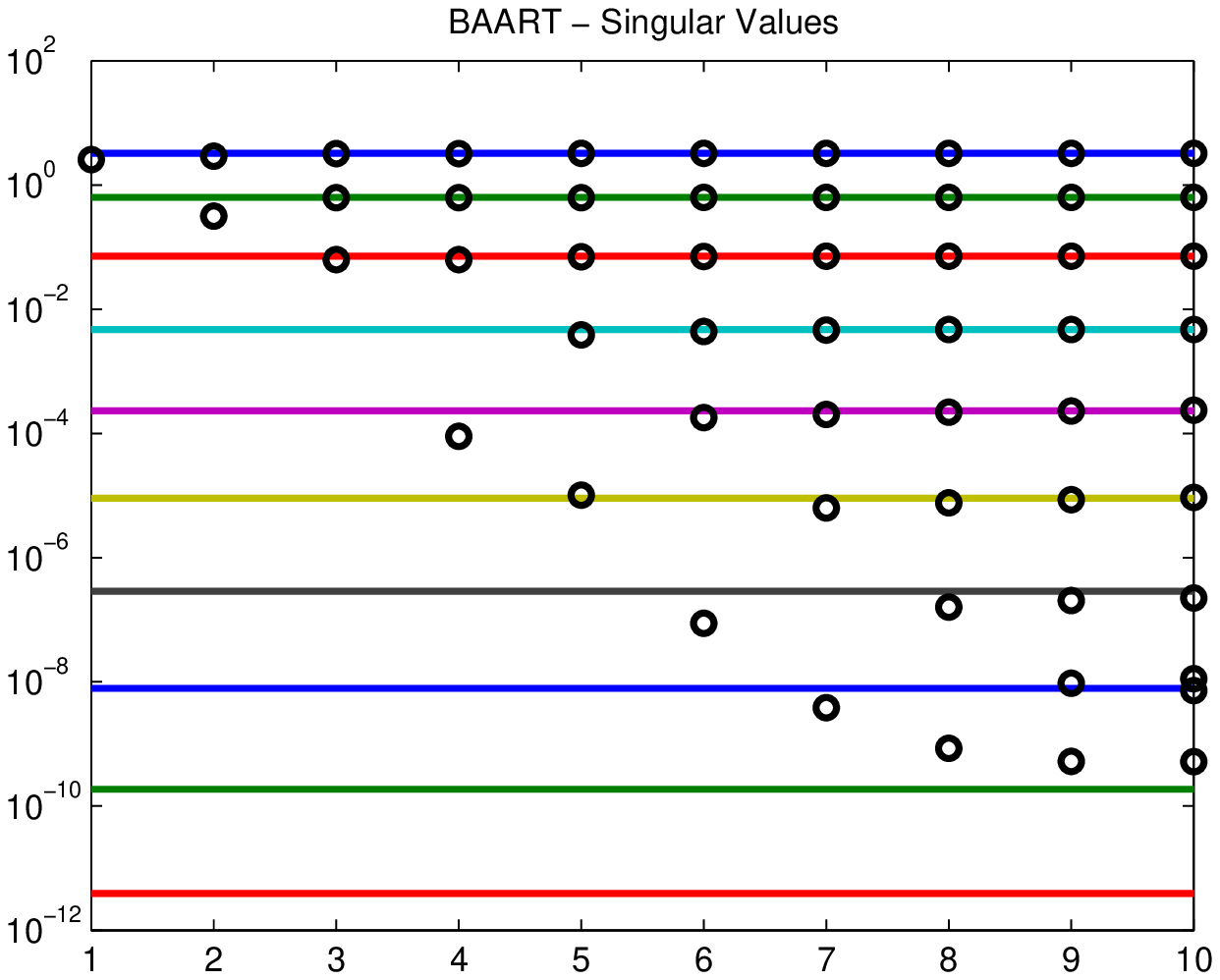} %
\includegraphics[width=0.45\textwidth]{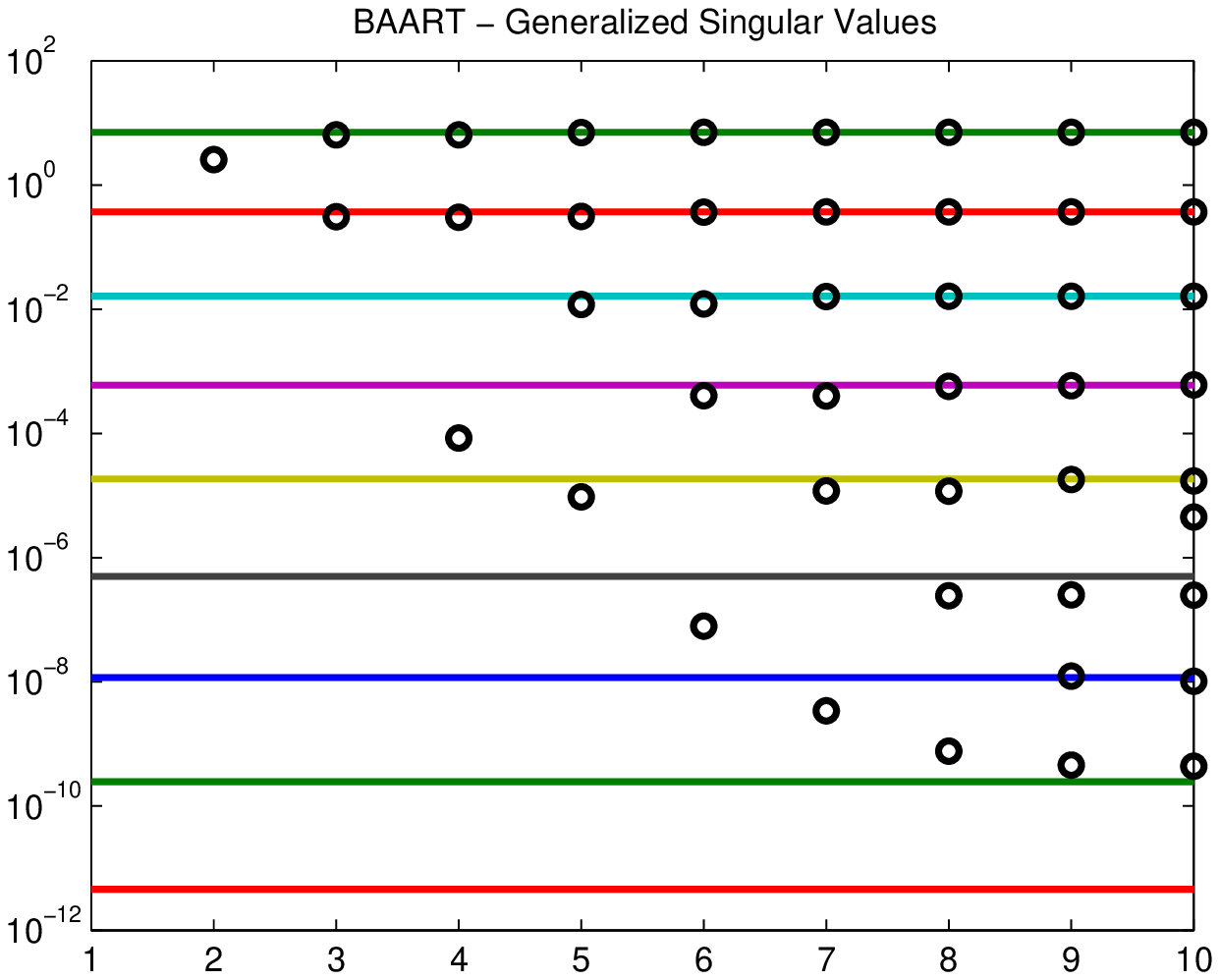} %
\includegraphics[width=0.45\textwidth]{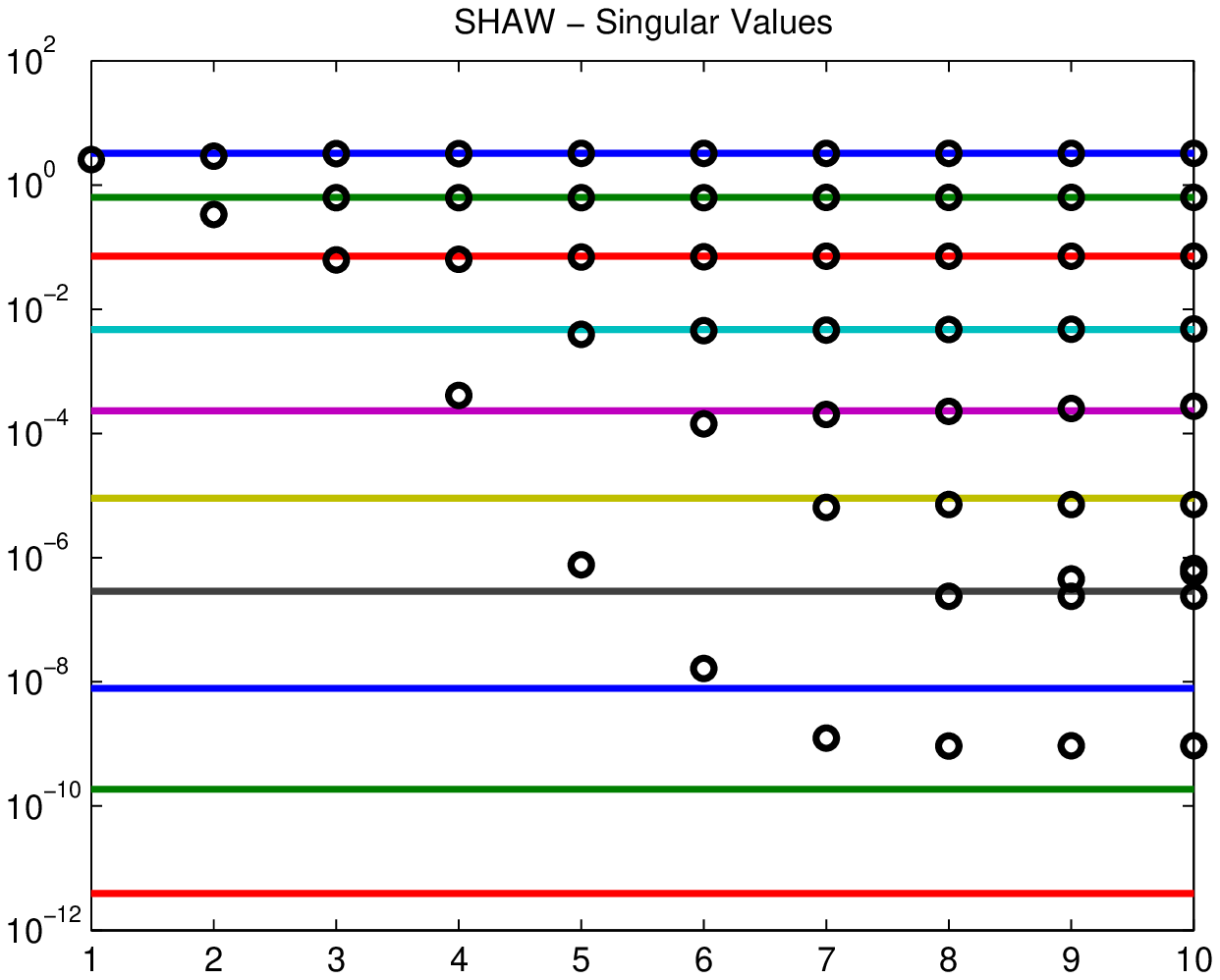} %
\includegraphics[width=0.45\textwidth]{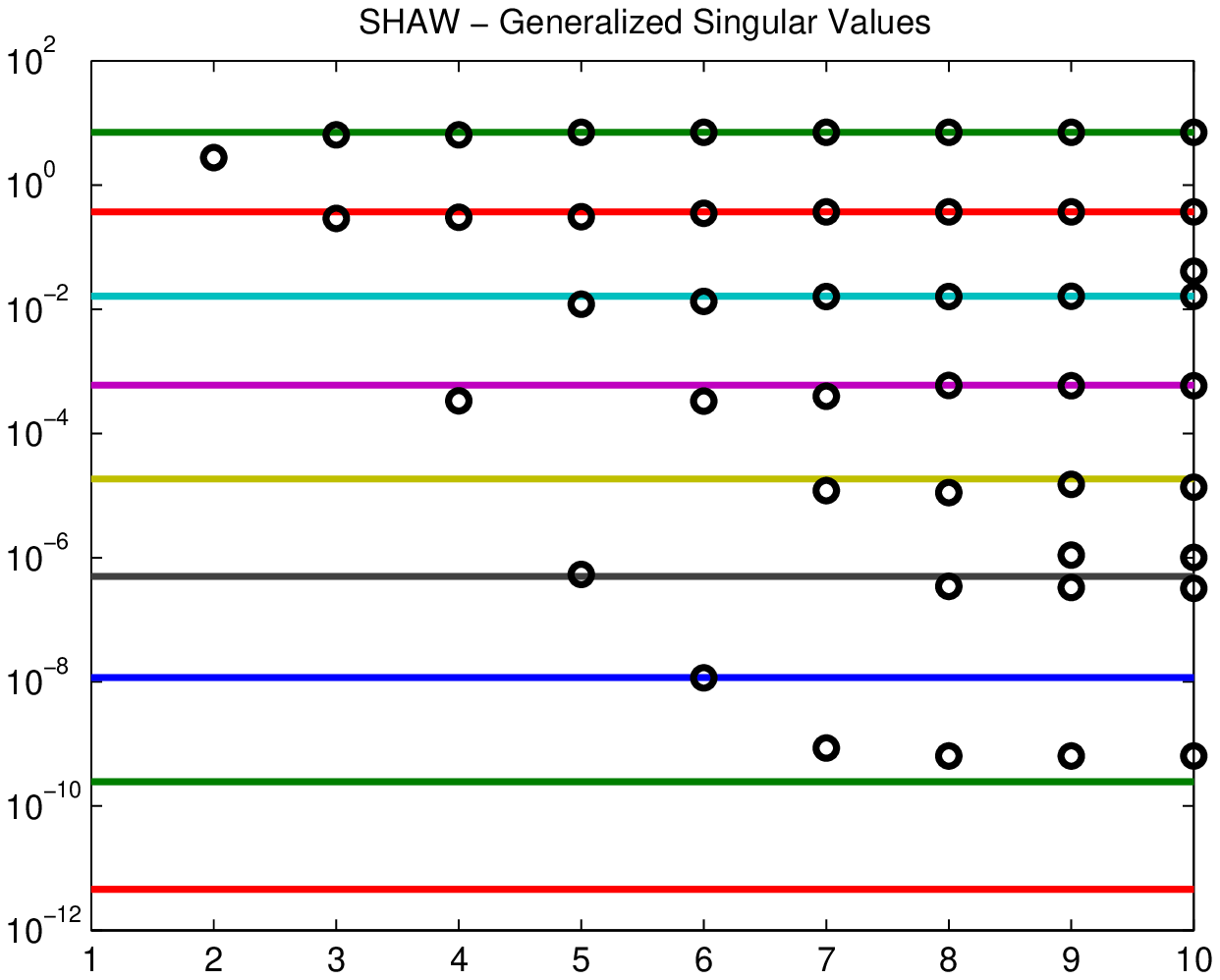}
\caption{\textit{Plots of the singular values (circle) of the matrix $%
\overline{H}_{m}$ (left) and the generalized singular values of the matrix
pair $(\overline{H}_{m},L_{m})$ (right) versus the iteration number k, for
the problem BAART and SHAW with $N=32 $. The solid lines represent the
singular values of the matrix A (left) and the generalized singular values
of the matrix pair $(A,L)$ (right).} }
\label{F3}
\end{figure}

\section{Generalized Cross-Validation}

A popular method for choosing the regularization parameter, which does non
require the knowledge of the noise properties nor its norm $\left\Vert
e\right\Vert $, is the Generalized Cross-Validation (GCV) \cite{Golub,Wab}.
The major idea of the GCV is that a good choice of $\lambda $ should predict
missing values, so that the model is not sensitive to the elimination of one
data point. This means that the regularized solution should predict a datum
fairly well, even if that datum is not used in the model. This viewpoint
leads to minimization with respect to $\lambda $ of the GCV function

\begin{equation*}
G(\lambda )=\frac{\Vert b-Ax_{\lambda }\Vert ^{2}}{[\text{\textrm{trace}}%
(I-AA_{\lambda })]^{2}},
\end{equation*}%
where $A_{\lambda }=(A^{T}A+\lambda ^{2}L^{T}L)^{-1}A^{T}$ is the matrix
that gives the regularized solutions of (\ref{tik}) from the normal equations%
\begin{equation*}
(A^{T}A+\lambda ^{2}L^{T}L)x_{\lambda }=A^{T}b.
\end{equation*}%
\

Using \ the GSVD of the matrix pair $(A,L)$, with a general $\ A\in \mathbb{R}%
^{M\times N},L\in \mathbb{R}^{P\times N}$, let $A=USX^{-1}$ and $L=VCX^{-1}$%
, where $S=diag(s_{1},...,s_{P})$ and $C=diag(c_{1},...,c_{P})$, the
generalized singular values $\gamma _{i}$ of $(A,L)$ are defined by the
ratios%
\begin{equation*}
\gamma _{i}=\frac{s_{i}}{c_{i}},\qquad i=1,...,P.
\end{equation*}%
Therefore, one can show that the expression of the GCV function is given by%
\begin{equation}
G(\lambda )=\frac{\sum_{i=1}^{N}\left( \frac{\lambda ^{2}}{\gamma
_{i}^{2}+\lambda ^{2}}u_{i}^{T}b\right) ^{2}}{\left( M-(N-P)-\sum_{i=1}^{P}%
\frac{\gamma _{i}^{2}}{\gamma _{i}^{2}+\lambda ^{2}}\right) ^{2}}.
\label{gcv1}
\end{equation}%
For the square case $M=N$, and $P\leq N$, rearranging the sum at the
denominator we obtain%
\begin{equation}
G(\lambda )=\frac{\sum_{i=1}^{N}\left( \frac{\lambda ^{2}}{\gamma
_{i}^{2}+\lambda ^{2}}u_{i}^{T}b\right) ^{2}}{\left( \sum_{i=1}^{P}\frac{%
\lambda ^{2}}{\gamma _{i}^{2}+\lambda ^{2}}\right) ^{2}}.  \label{gcv2}
\end{equation}%
The GCV criterion is then based on the choice of $\lambda $ which minimizes $%
G(\lambda )$. It is well known that this minimization problem is generally
ill-conditioned, since the function $G(\lambda )$ is typically flat in a
relatively wide region around the minimum. As a consequence, this criterion
may even lead to a poor regularization \cite{Kon,Luk,Thom}.

As already said in the Introduction, for large-scale problems the GCV
approach for (\ref{tik}) is too much expensive since it requires the SVD
(GSVD). In this setting, our idea is to fully exploit the approximation
properties of the Arnoldi algorithm investigated in Section 3 and 4. In
particular, our aim is to define a sequence of regularization parameters $%
\left\{ \lambda _{m}\right\} $, i.e., one for each iteration of the Arnoldi
algorithm, obtained by the minimization of the following GCV function
approximations%
\begin{equation}
G_{m}(\lambda )=\frac{\left\Vert \overline{H}_{m}y_{m,\lambda }-\left\Vert
b\right\Vert e_{1}\right\Vert ^{2}}{\left( N-m+\sum_{i=1}^{m}\frac{\lambda
^{2}}{\gamma _{i}^{(m)2}+\lambda ^{2}}\right) ^{2}}, \label{gcvm}
\end{equation}%
where $y_{m,\lambda }$ solves the reduced minimization (\ref{rm}), and $\gamma
_{i}^{(m)}$, $i=1,...,m$, are the approximations of the generalized singular
values, obtained with the Arnoldi process. Note that
\begin{equation*}
\left\Vert \overline{H}_{m}y_{m,\lambda }-\left\Vert b\right\Vert
e_{1}\right\Vert ^{2}=\sum_{i=1}^{m}\left( \frac{\lambda ^{2}}{\gamma
_{i}^{(m)2}+\lambda ^{2}}u_{i}^{(m)^{T}}c\right) ^{2}+\left(
u_{m+1}^{(m)^{T}}c\right) ^{2},
\end{equation*}%
where $u_{i}^{(m)}$ is defined as in Proposition \ref{p2} and $c=\left\Vert
b\right\Vert e_{1}$, so that the construction of $G_{m}(\lambda )$ can be
obtained working in reduced dimension. The basic idea which leads to the
approximation $G_{m}(\lambda )\approx G(\lambda )$, is to set equal to $0$
the generalized singular values that are not\ approximated by the Arnoldi
algorithm, and that are expected to be close to $0$ after few iterations.\
This is justified by the analysis and the experiments of Section 3 and 4.

We remark that in a hybrid approach \cite{KO}, one aims to regularize the
projected problem
\begin{equation}
\min_{y\in \mathbb{R}^{m}}\left\{ \left\Vert \overline{H}_{m}y-\left\Vert
b\right\Vert e_{1}\right\Vert \right\} .  \label{rpm}
\end{equation}%
Since no geometrical information on the solution of (\ref{rpm}) can be
inherited from the solution of the original problem, the choice of $%
L_{m}=I_{m}$ as regularization operator is somehow forced (this is a
standard strategy for hybrid methods \cite[\S 6.7]{PCH}). In this framework,
if the GCV criterion is used to regularize (\ref{rpm}), the basic
difference with respect to (\ref{gcvm}) is at the denominator, where $N-m$
is replaced by $m$. We observe moreover that (\ref{gcvm}) is similar to the
GCV approximation commonly used for iterative methods, in which the
denominator is simply $N-m$ \cite[\S 7.4]{PCH}.

In the following, the algorithm that has been used for the tests of the
next sections.

\begin{center}
\hrulefill \\[0pt]
\vspace{-0.1cm} \textbf{AT - GCV Algorithm } \\[0pt]
\vspace{-0.2cm} \hrulefill \\[0pt]
\begin{tabbing}
NN \= XX \= XX \= XX \= XX \= \kill
{\bf given} $A\in {\mathbb{R}}^{N\times N},\;b\in {\mathbb{R}}^{N}$, $\delta$\\
{\bf while} $\big| \left\Vert r_m \right\Vert- \left\Vert r_{m-1}\right\Vert \big| /\left\Vert r_m \right\Vert  \geq \delta$\\
\> {update} $\overline{H}_{m}$ and $L_{m}$ from (\ref{hm}) and (\ref{nlm})\\
\> {compute} GSVD($\overline{H}_m,L_m $)\\
\> {compute} $\lambda_m=\arg \min_{\lambda} G_{m}(\lambda $)\\
\> {solve}
$
\min_{y_{m}\in \mathbb{R}^{m}}\left\Vert \left(\begin{array}{c}
\overline{H}_{m} \\
\lambda_m L_{m}
\end{array}
\right) y_{m}-\left(
\begin{array}{c}
\left\Vert b\right\Vert e_{1} \\
0 \end{array}
\right) \right\Vert ^{2}$\\
\>{compute} the corresponding residual $r_{m}$ \\
{\bf end}\\
{compute} $x_{m}=W_{m}y_{m}$
\end{tabbing}
\vspace{-0.5cm} 
\hrulefill
\end{center}


The stopping rule used in the algorithm is just based on the residual. As an alternative, one may even employ the strategy adopted in \cite{Nagy}, based on the observation of the GCV approximations.

\section{Numerical results}

In order to test the performance of the proposed method, we consider again
some classical test problems taken from the Regularization Tools \cite{H1}.
In particular in Figures \ref{F4}-\ref{F5}, we consider the problems BAART,
SHAW, FOXGOOD, I\_LAPLACE, with right-hand side affected by 0.1\% or 1\%
Gaussian noise. The regularization operator is always the discretized first
derivative, augmented with a zero row at the bottom to make it square (cf. (%
\ref{nlm})). For each experiment we show: (a) the approximation of $%
G(\lambda )$ obtained with the functions $G_{m}(\lambda )$ for some values
of $m$, with a graphical comparison of the local minima; (b) the approximate
solution; (c) the relative residual and error history; (d) the sequence of
selected parameters $\left\{ \lambda _{m}\right\} $, with respect to the one
obtained with the minimization of $G(\lambda )$ (denoted by $\lambda_A$ in the pictures)
and the optimal one ($\lambda_{opt}$) obtained
by the minimization of the distance between the regularized and the true
solution \cite{OL1}
\begin{equation*}
\min_{\lambda } \left\Vert x_{reg}-x_{true}\right\Vert
^{2}\equiv \min_{\lambda }f(\lambda ),
\end{equation*}%
where

\begin{equation*}
f(\lambda )=\left\{ \sum_{i=1}^{p}\left( \frac{\lambda ^{2}}{(\gamma
_{i}^{2}+\lambda ^{2})}\frac{u_{i}^{T}b}{\sigma _{i}}x_{i}-%
\sum_{i=p+1}^{N}(u_{i}^{T}b)x_{i}\right) -\sum_{i=1}^{N}\frac{u_{i}^{T}b}{%
\sigma _{i}}v_{i}\right\} ^{2}.
\end{equation*}

\begin{figure}[H]
\centering
\includegraphics[width=0.49\textwidth]{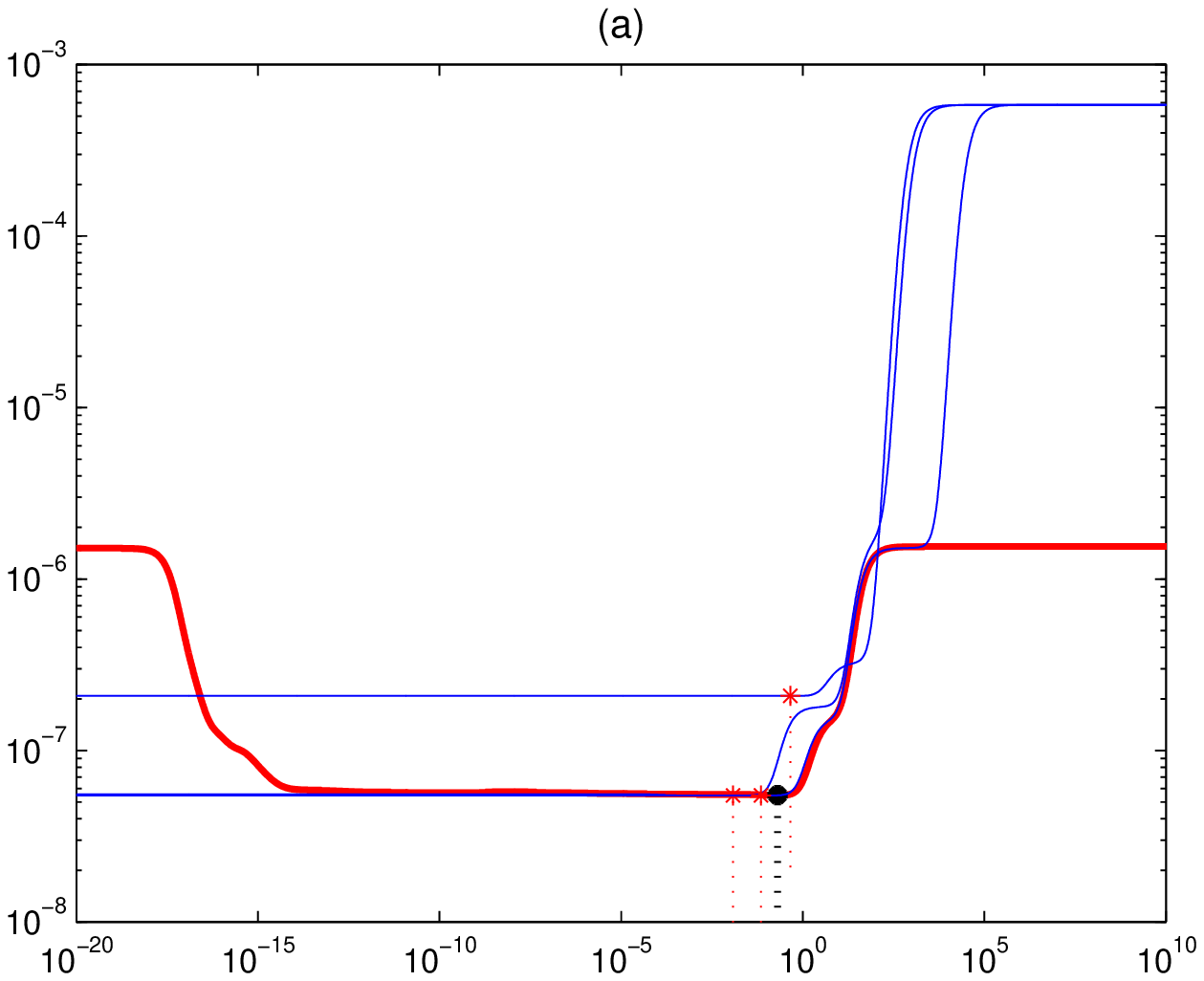} %
\includegraphics[width=0.49\textwidth]{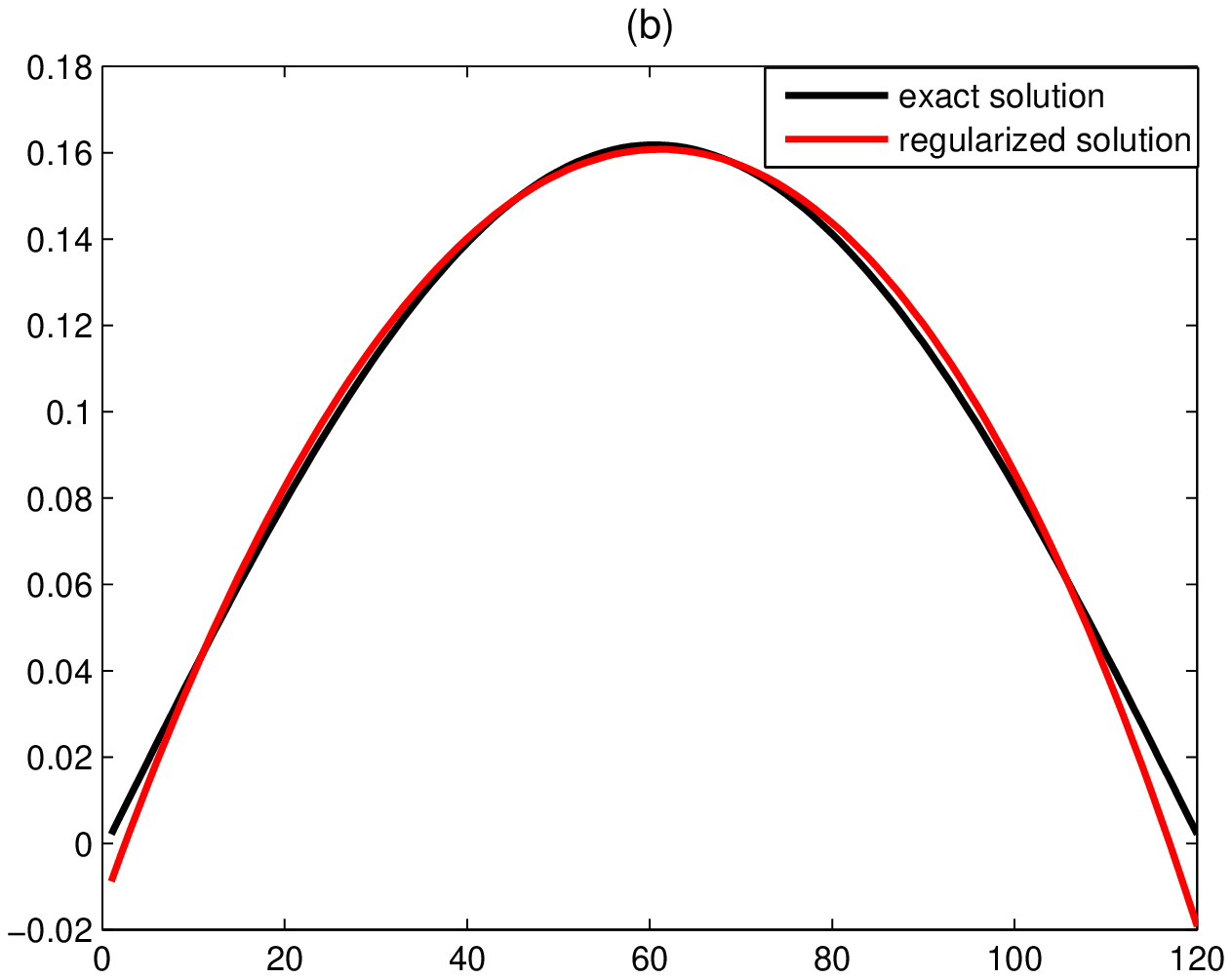} \newline
\includegraphics[width=0.49\textwidth]{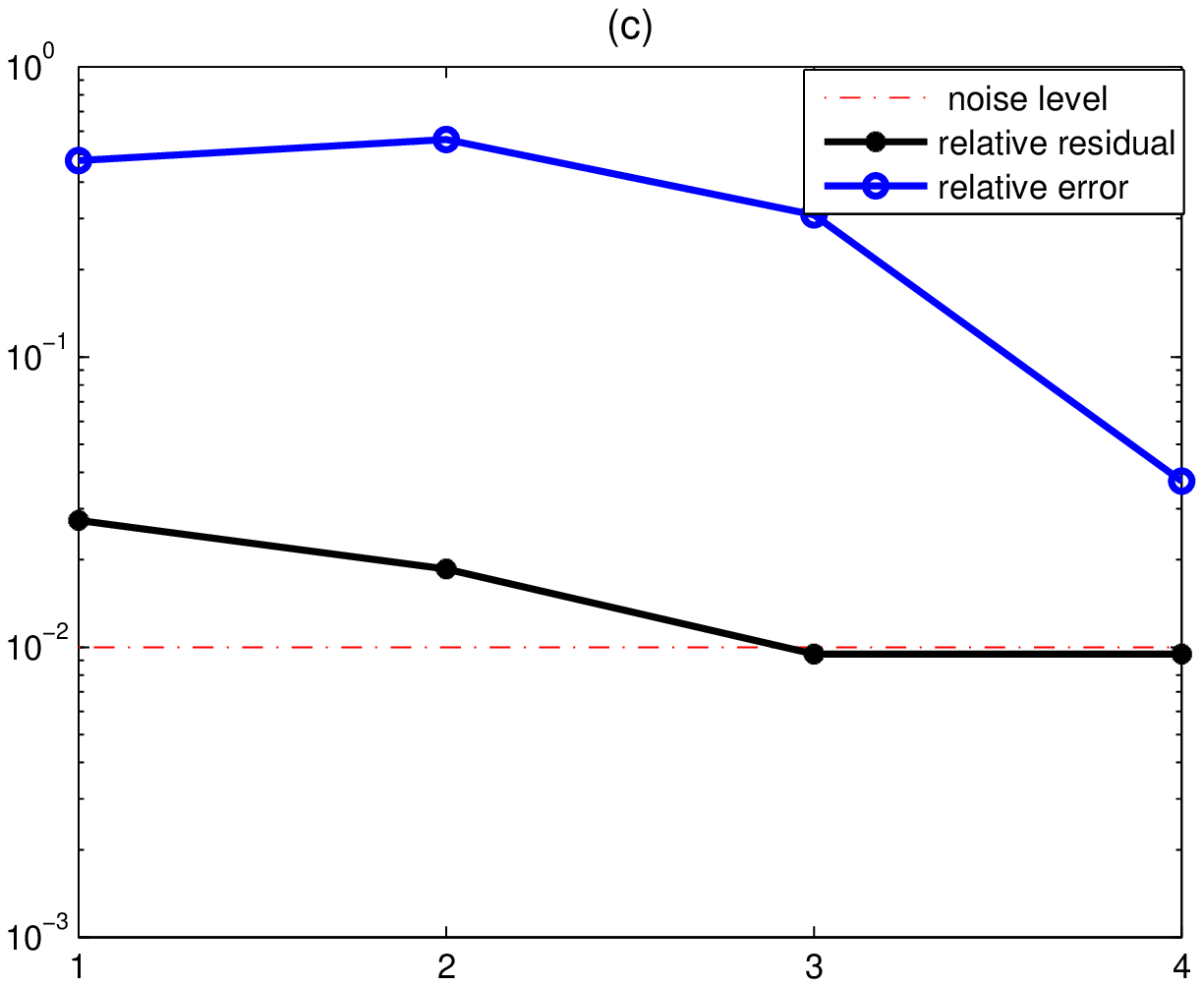} %
\includegraphics[width=0.49\textwidth]{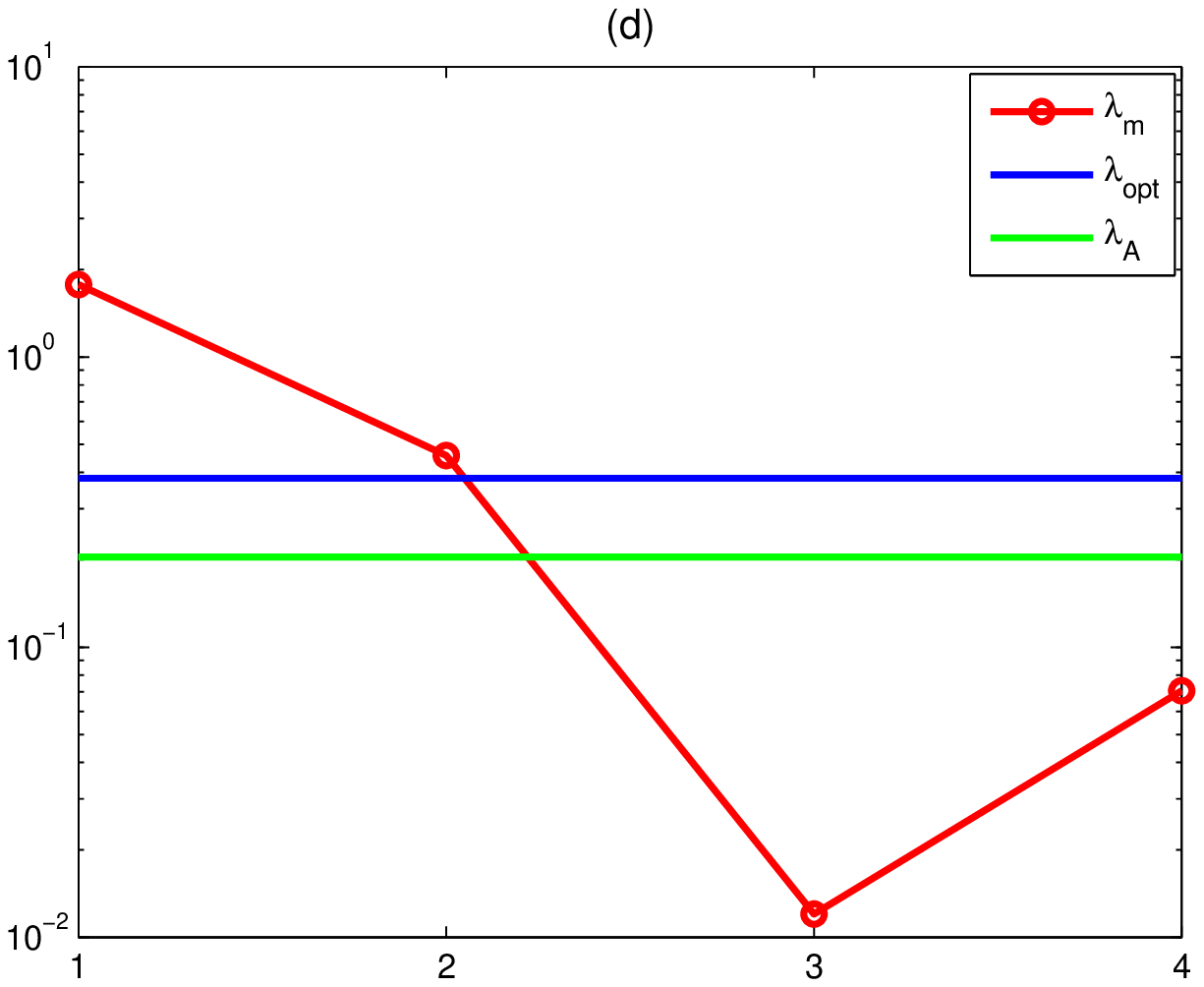}
\includegraphics[width=0.49\textwidth]{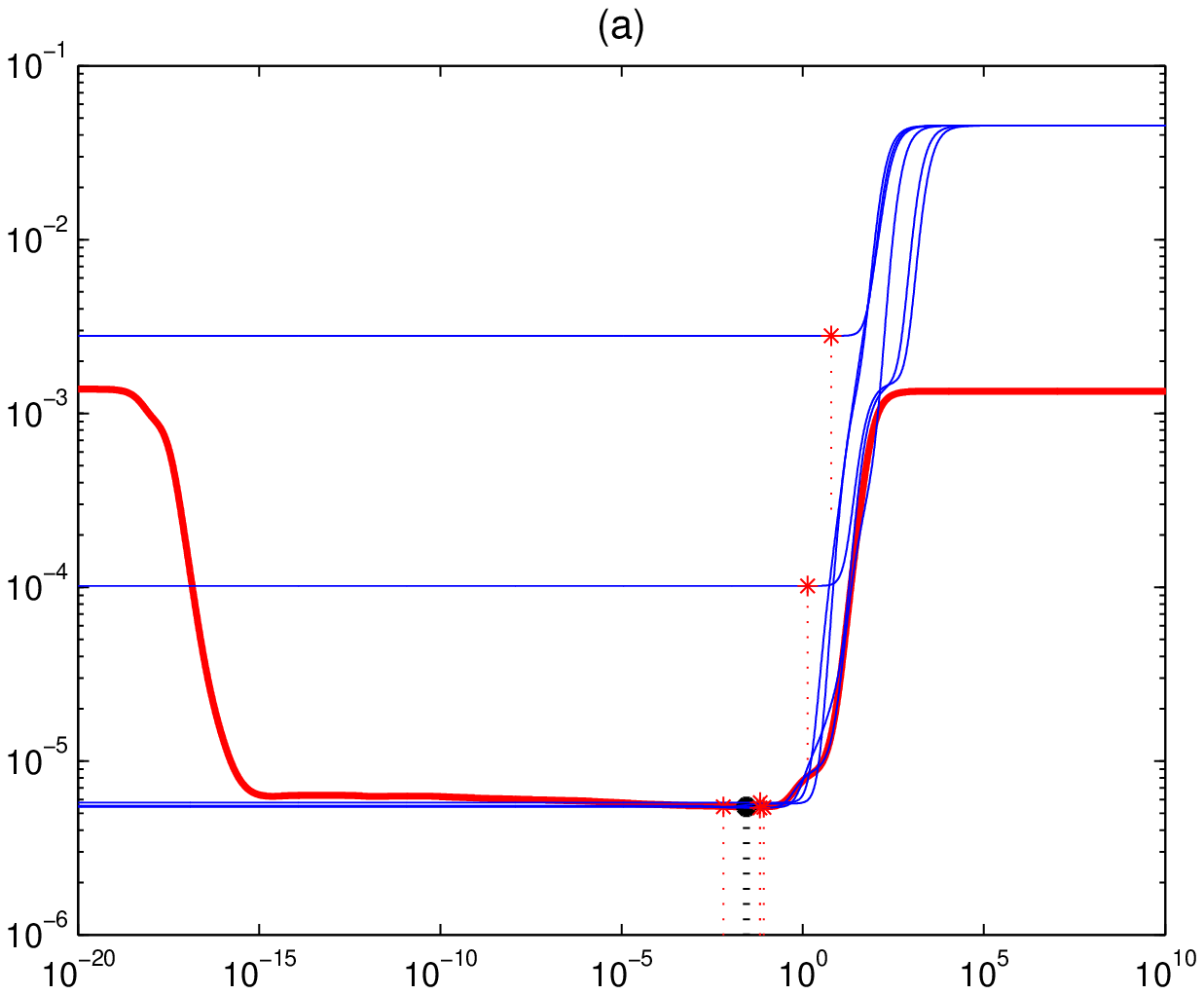} %
\includegraphics[width=0.49\textwidth]{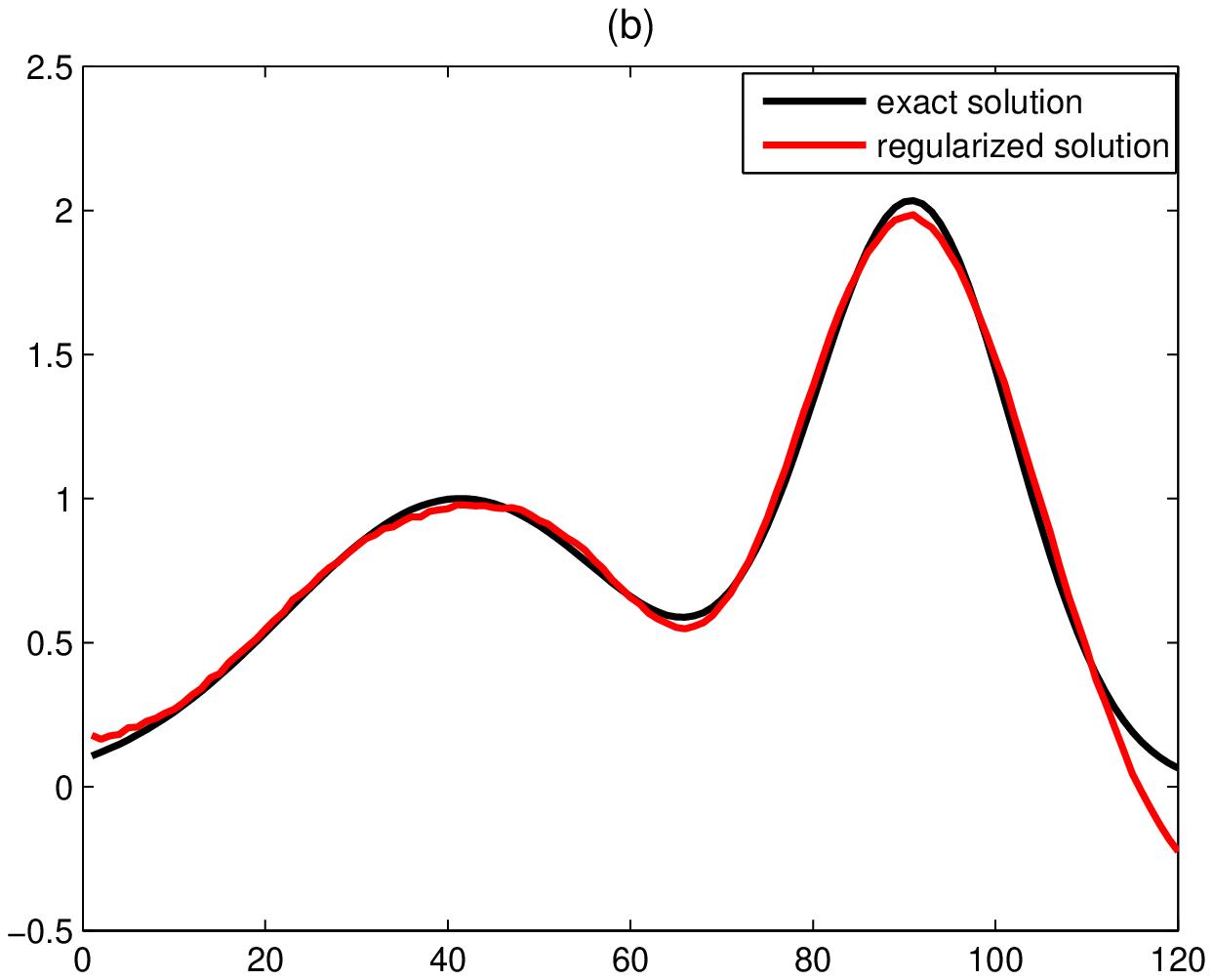} \newline
\includegraphics[width=0.49\textwidth]{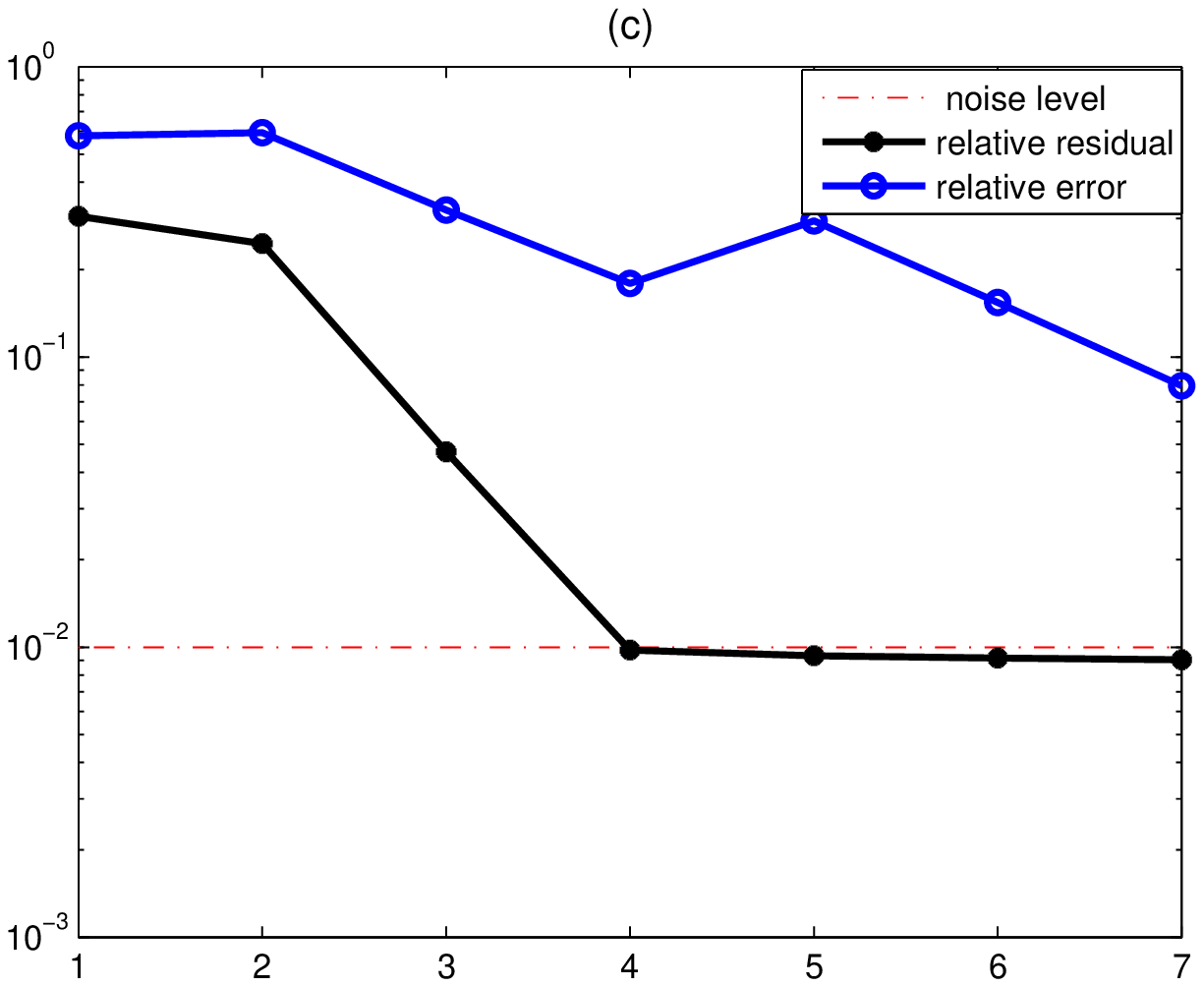} %
\includegraphics[width=0.49\textwidth]{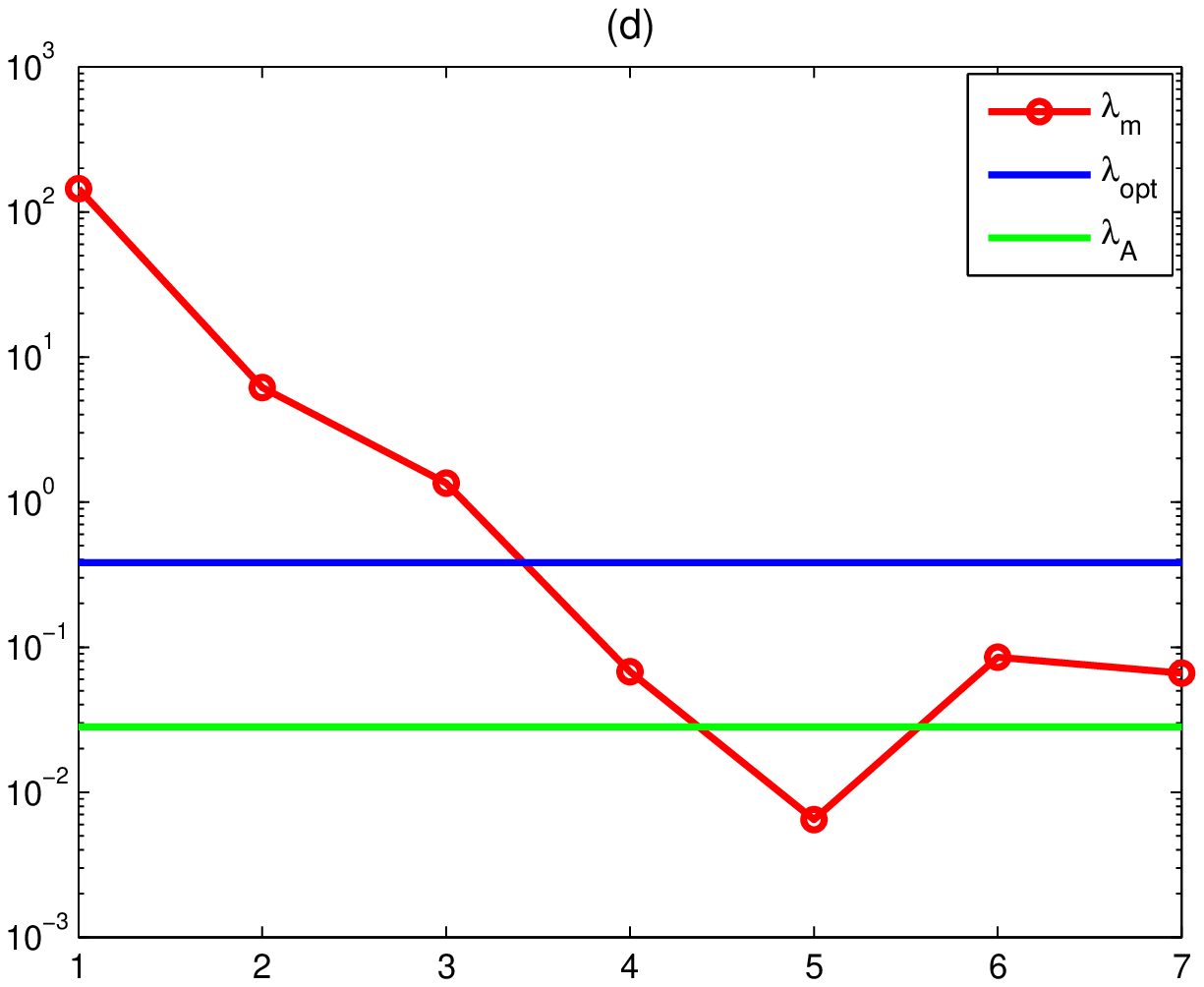}
\caption{\textit{Results for BAART (top) and SHAW (bottom). The dimension of each problem is $N=120$. Noise level $\varepsilon =10^{-2}.$}}
\label{F4}
\end{figure}

\begin{figure}[H]
\centering
\includegraphics[width=0.49\textwidth]{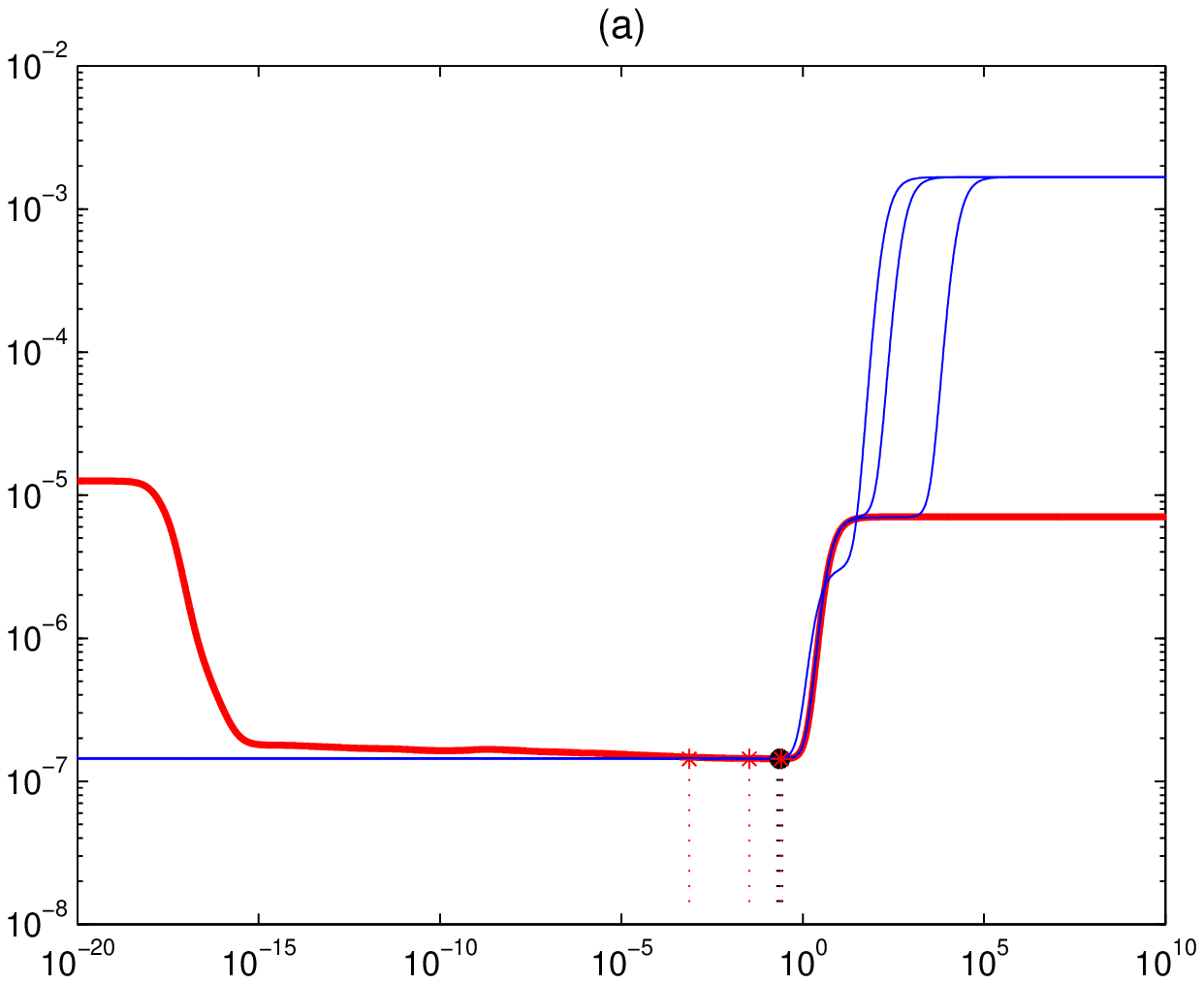} %
\includegraphics[width=0.49\textwidth]{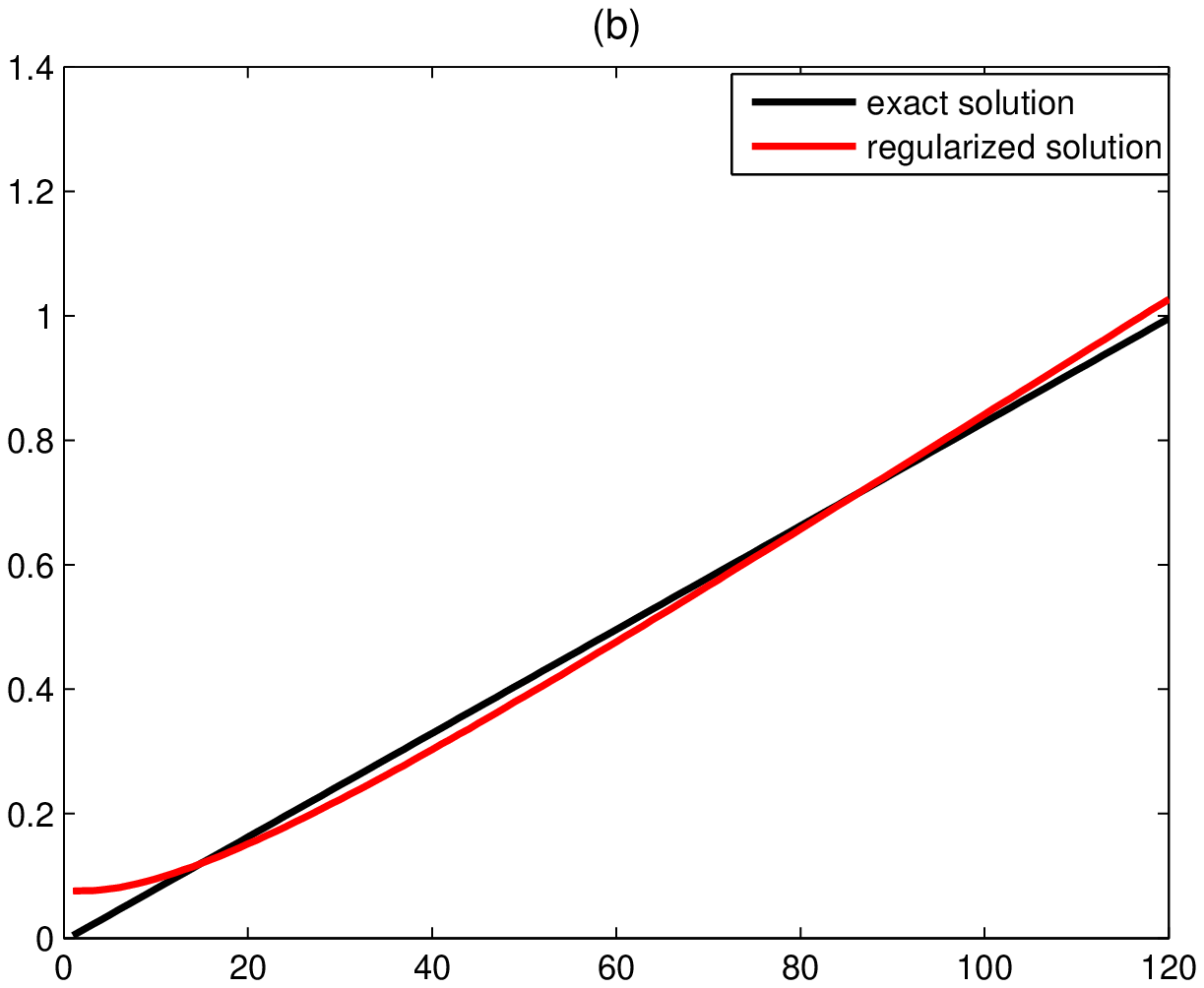} \newline
\includegraphics[width=0.49\textwidth]{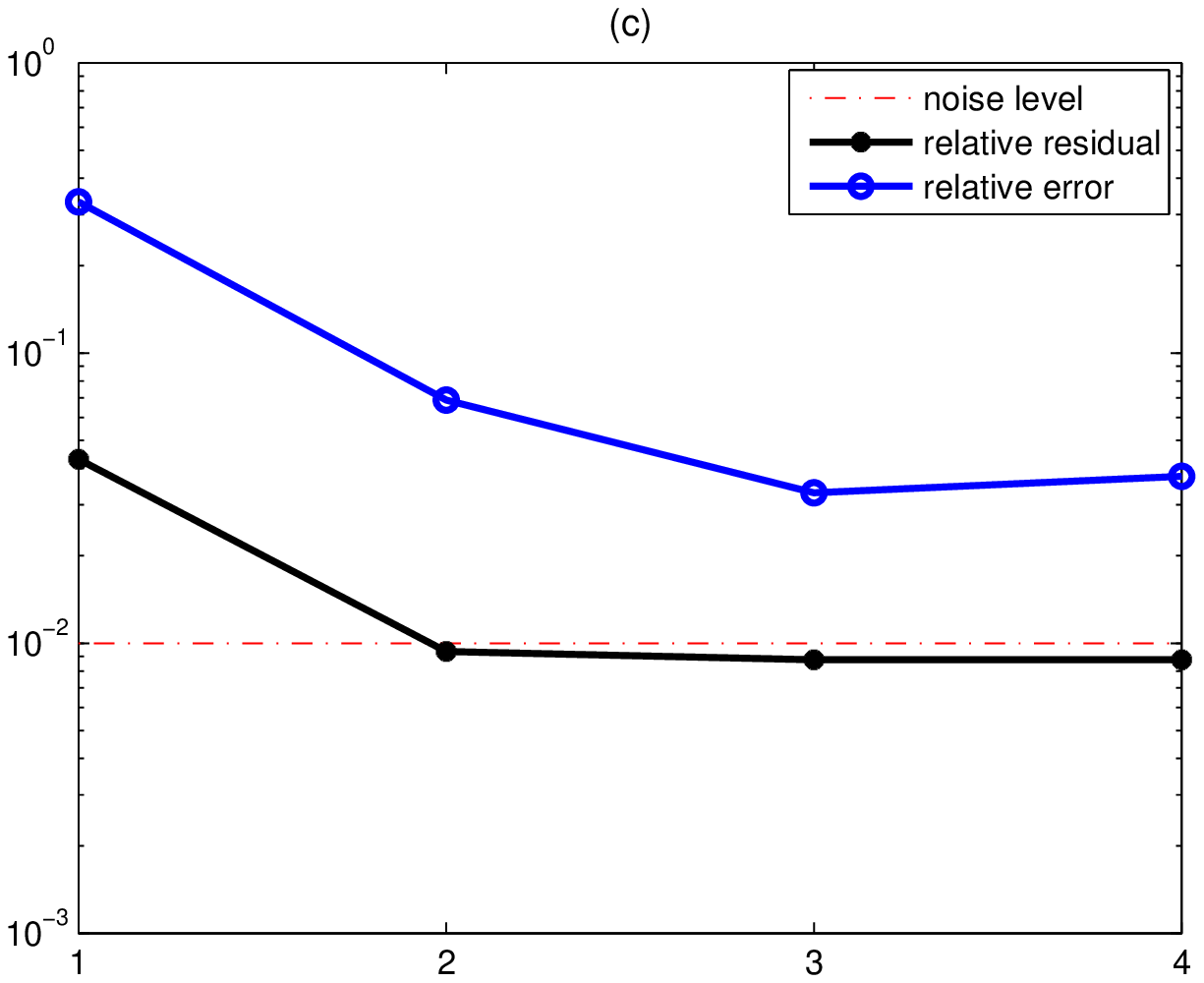} %
\includegraphics[width=0.49\textwidth]{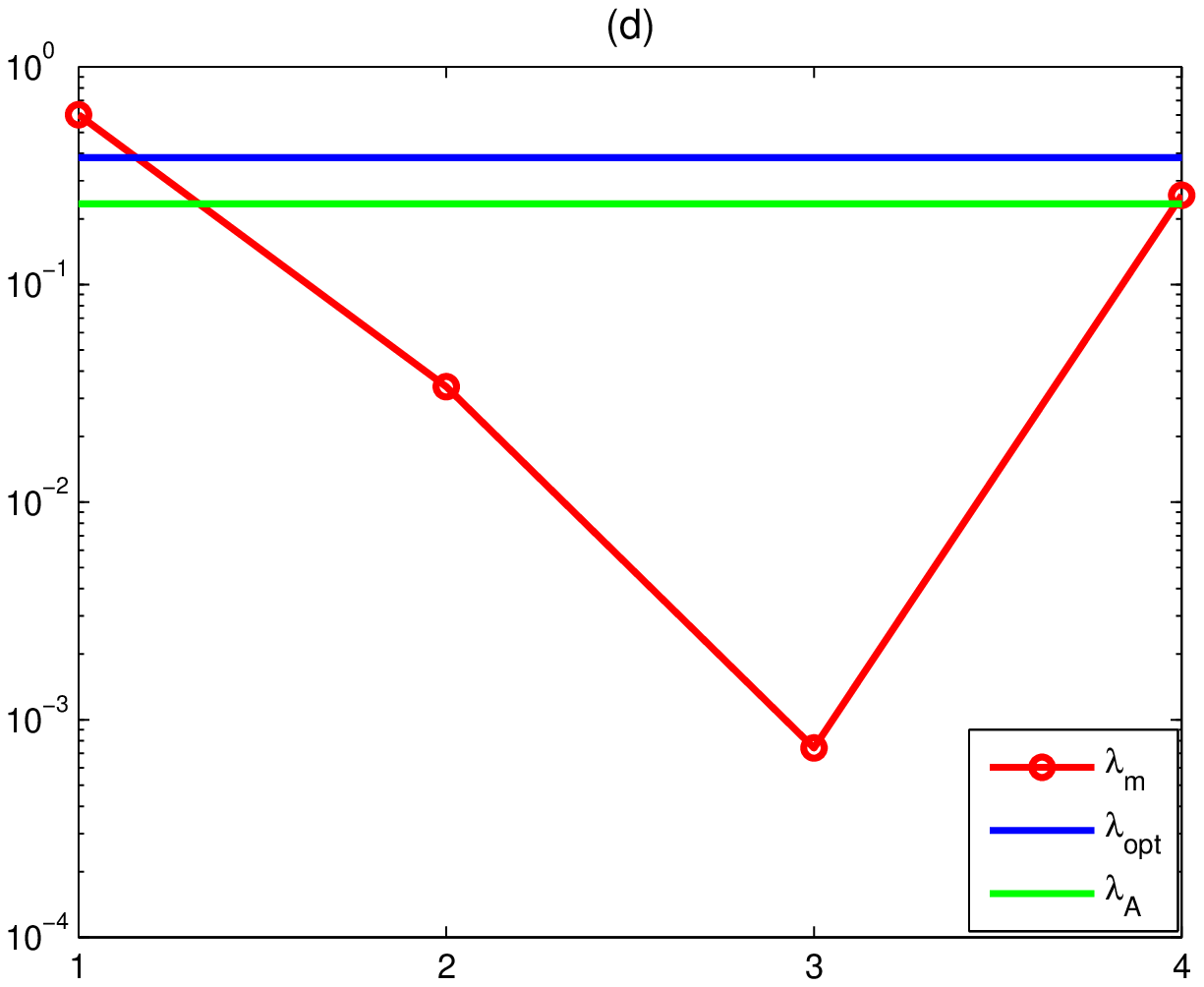} \newline
\par
\includegraphics[width=0.49\textwidth]{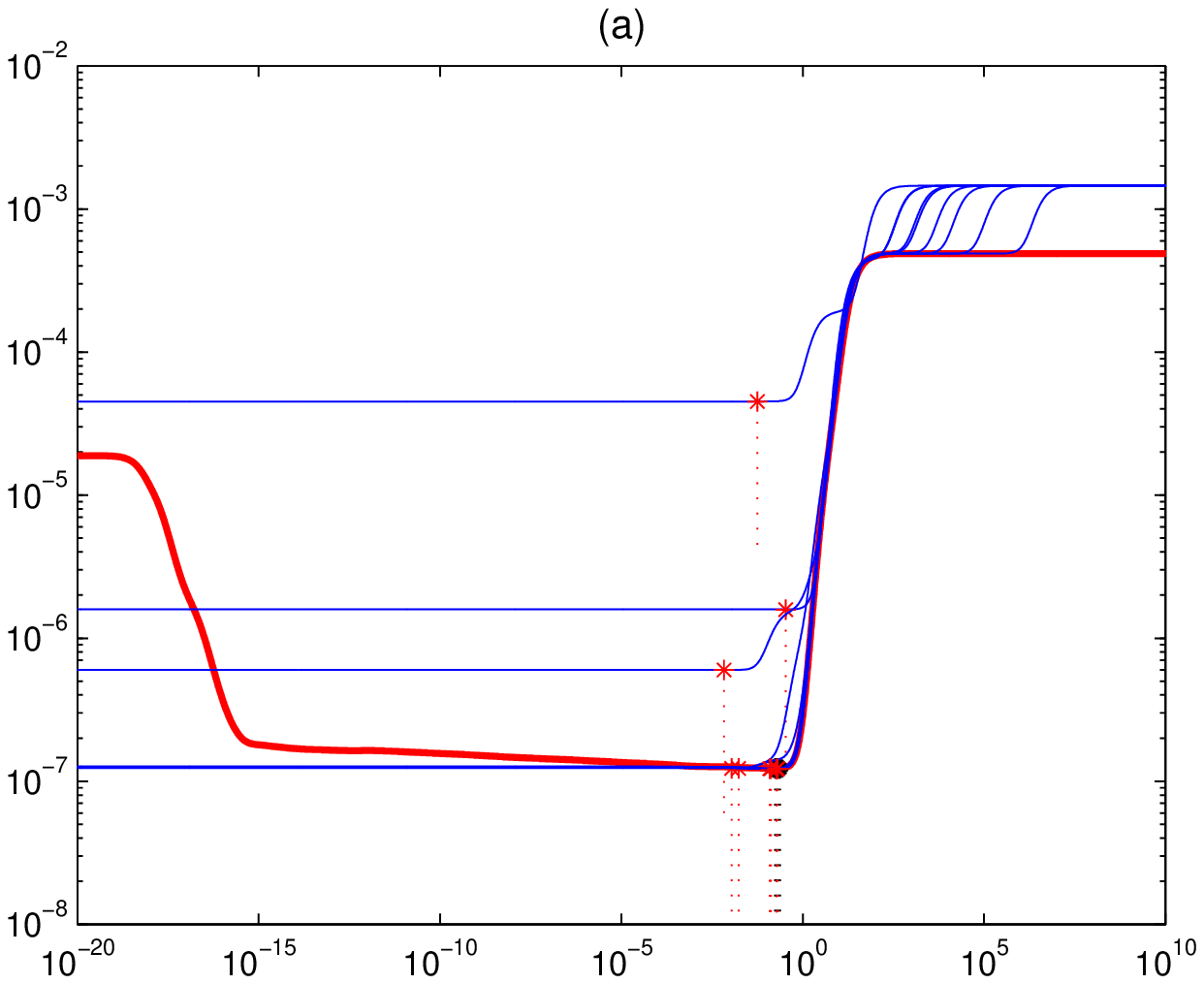} %
\includegraphics[width=0.49\textwidth]{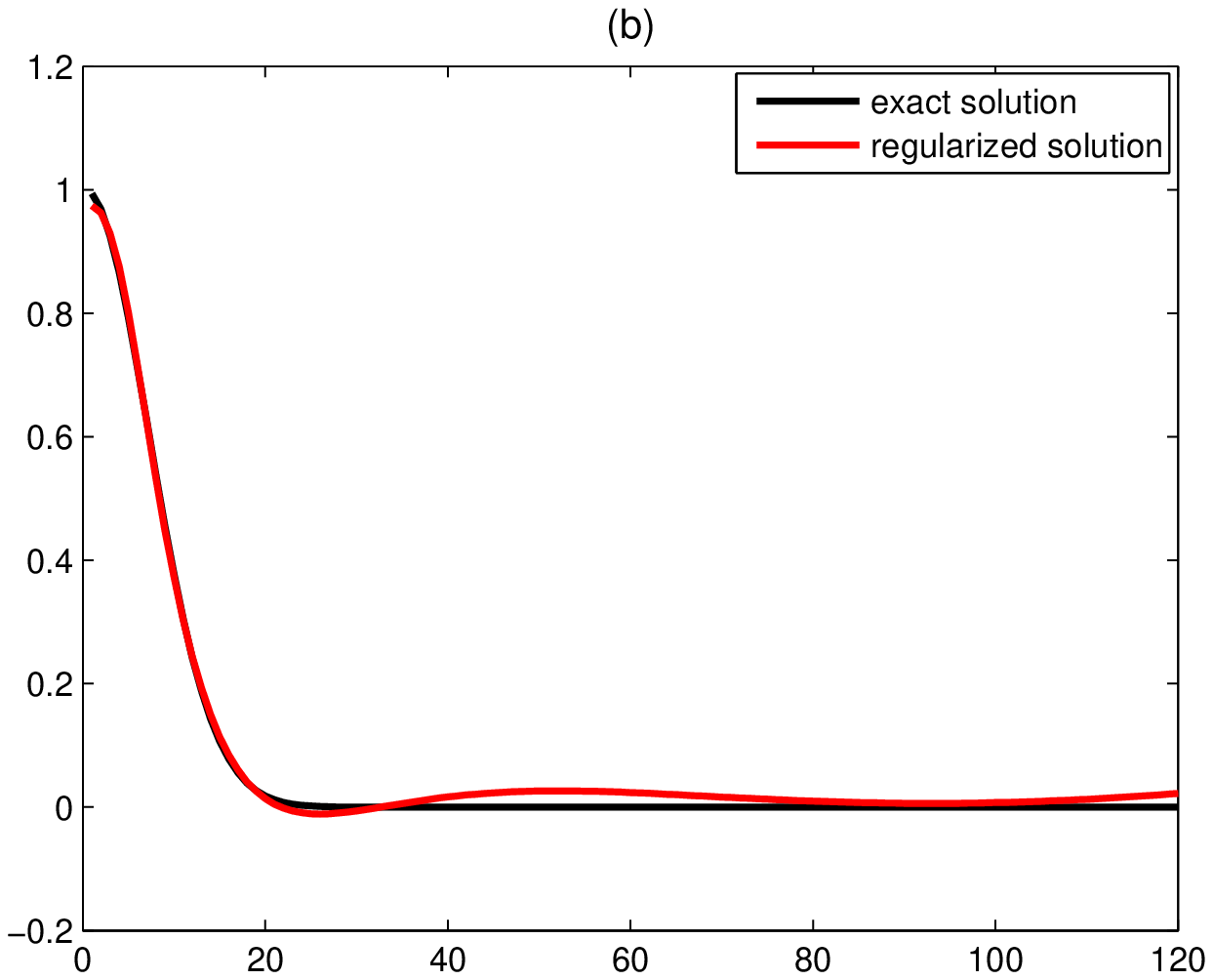} \newline
\includegraphics[width=0.49\textwidth]{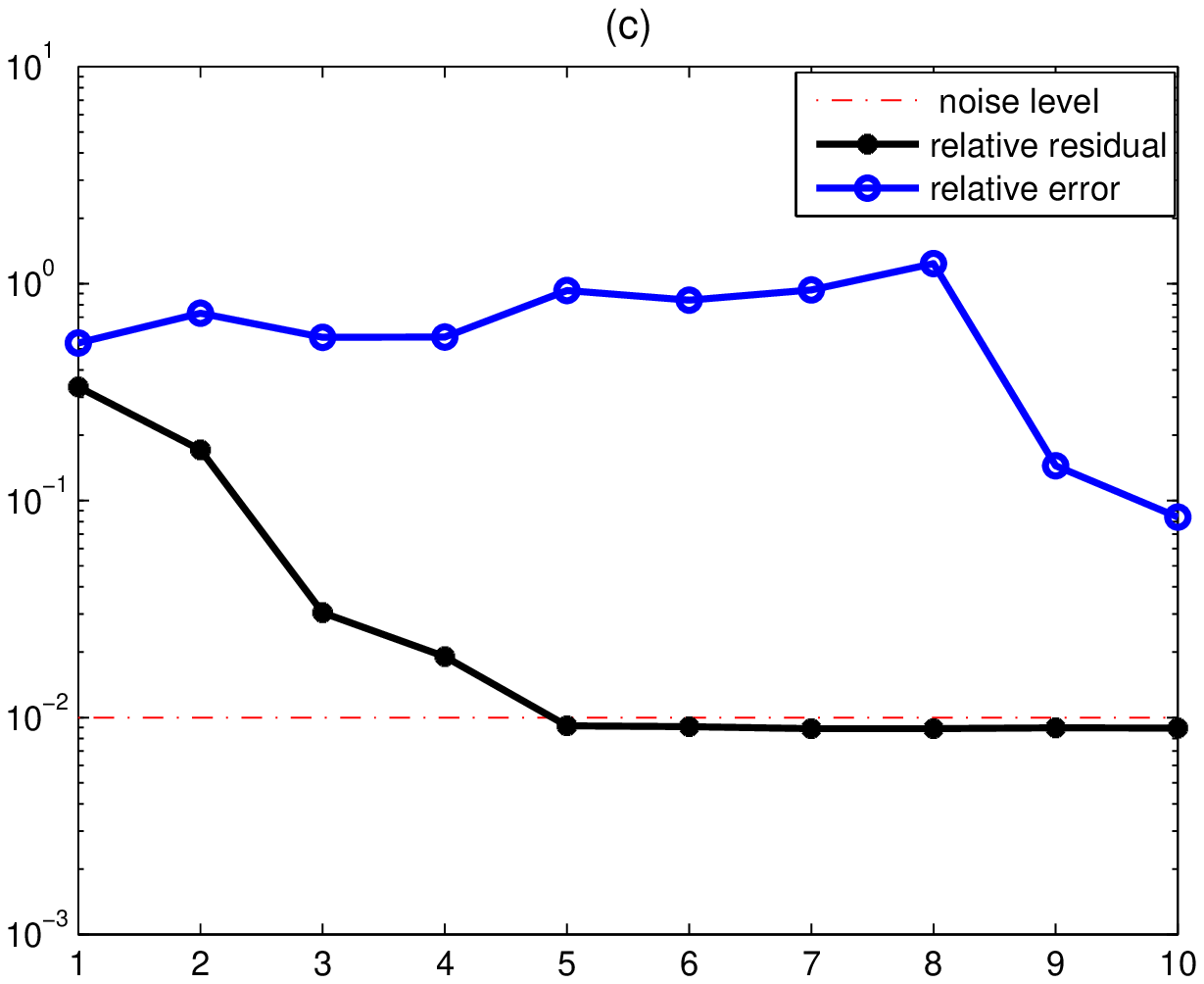} %
\includegraphics[width=0.49\textwidth]{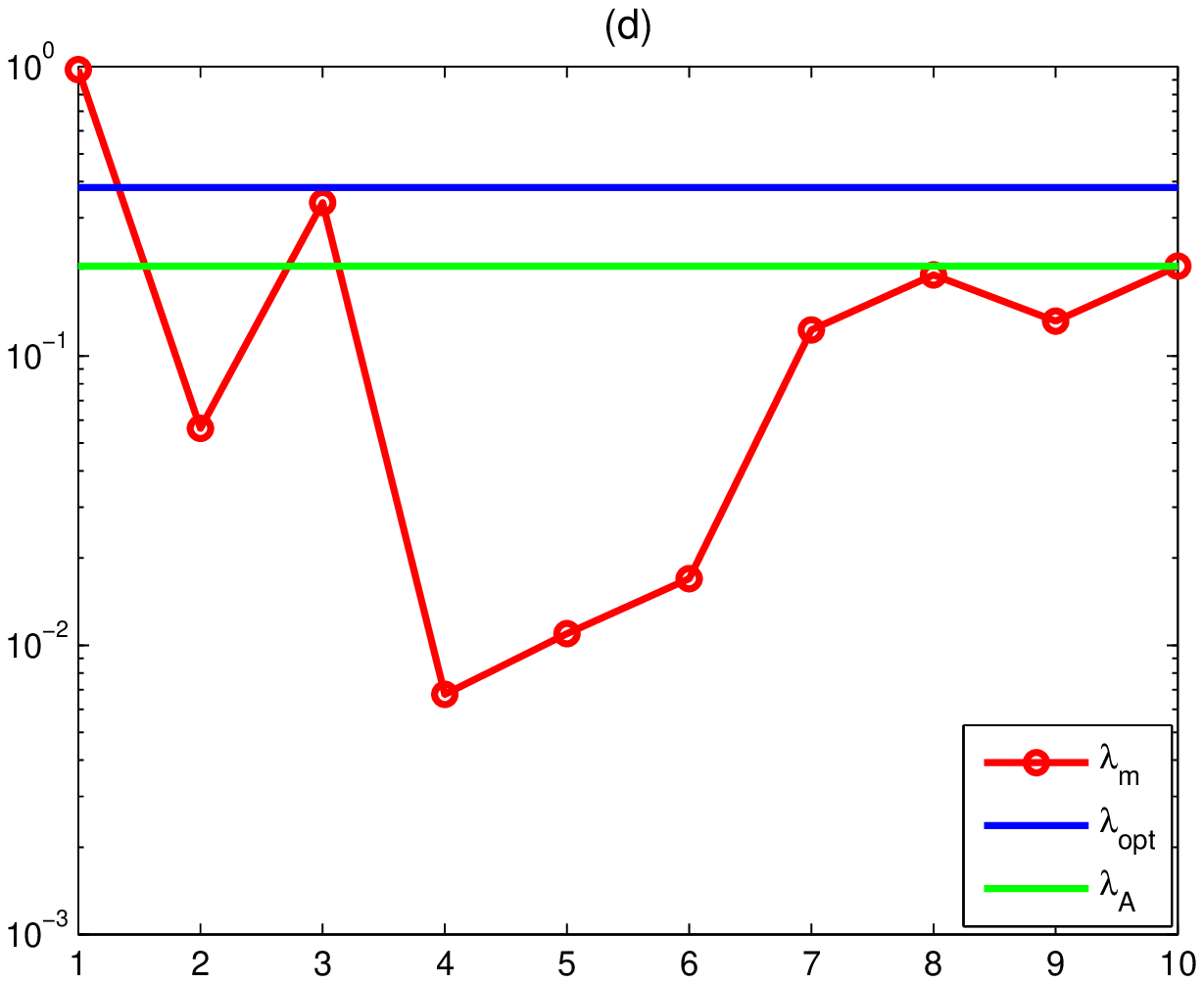} \newline
\caption{\textit{Results for FOXGOOD (top) and I\_LAPLACE (bottom). The dimension of each problem is $N=120$. Noise level $\varepsilon =10^{-3}.$}}
\label{F5}
\end{figure}

\section{An example of image restoration}

We conclude with an illustration of the performance of the GCV-Arnoldi
approach on a 2D image deblurring problem which consist in recovering the
original $n\times n$ image from a blurred and noisy observed image.

Let $X$ \ be a $n\times n$ two dimensional image. The vector $\overline{x}$
of dimension $N=n^{2}$ obtained by stacking the columns of the image $X$
represents a blur-free and noise-free image. We generate an associated
blurred and noise-free image $\overline{b}$ by multiplying $\overline{x}$ by
a block Toeplitz matrix $A\in \mathbb{R}^{N\times N}$ with Toeplitz blocks,
implemented in the function \texttt{blur.m} from the Regularization Tools
\cite{H1}. This Matlab function has two parameters, \texttt{band} and
\texttt{sigma}; the former specifies the half-bandwidth of the Toeplitz
blocks and the latter the variance of the Gaussian point spread function.
The blur and noise contaminated image $b\in \mathbb{R}^{N}$ is obtained by
adding a noise-vector $e\in \mathbb{R}^{N}$, so that $b=A\overline{x}+e$. We
assume the blurring operator $A$ and the corrupted image $b$ to be available
while no information is given on the error $e$, we would like to determine a
restoration which accurately approximates the blur-free and noise-free image
$\overline{x}$.

We consider the restoration of a corrupted version of the $256\times 256$
test image \texttt{mri.png}. Contamination is by 1\% white Gaussian
noise and space-invariant Gaussian blur. The latter is generated as
described above with blur parameters \texttt{band}=7, \texttt{sigma}=2, so
that the condition number of $A$ is around $10^{13}$. Figure \ref{F6}
displays the performance of the AT-GCV. On the left the blur-free and
noise-free image, on the middle the corrupted image, on the right the
restored image. From top to bottom the image in original size and two
different zooms. The regularization operator is defined as (cf. \cite{GN})
\begin{equation*}
L=I_{n}\otimes L_{1}+L_{1}\otimes I_{n}\in \mathbb{R}^{N\times N},
\end{equation*}%
where $L_{1}\in \mathbb{R}^{n\times n}$ is the discretized first derivative
with a zero row at the bottom (cf. also \cite[\S 5]{KHE}).
\begin{figure}[H]
\centering
\includegraphics[width=0.32\textwidth]{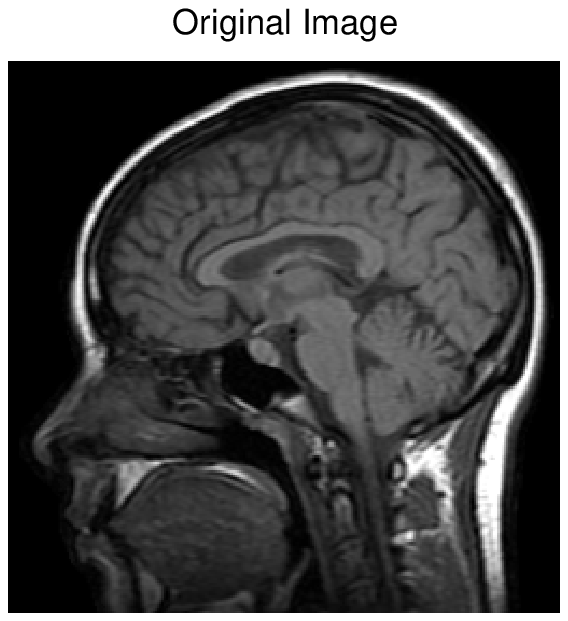} %
\includegraphics[width=0.32\textwidth]{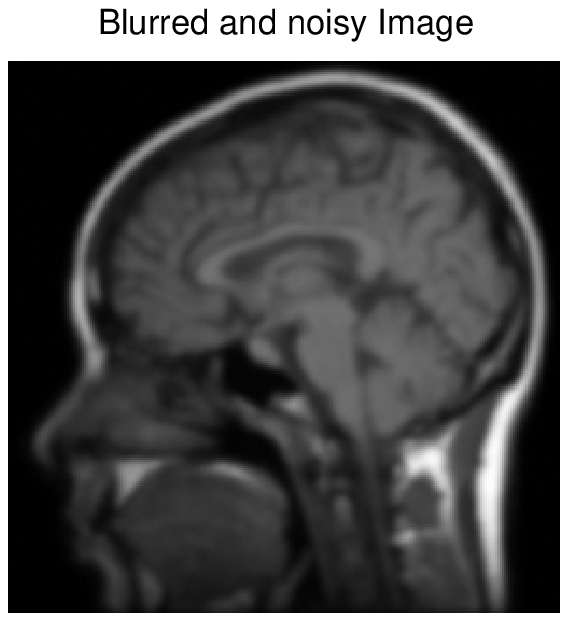} %
\includegraphics[width=0.32\textwidth]{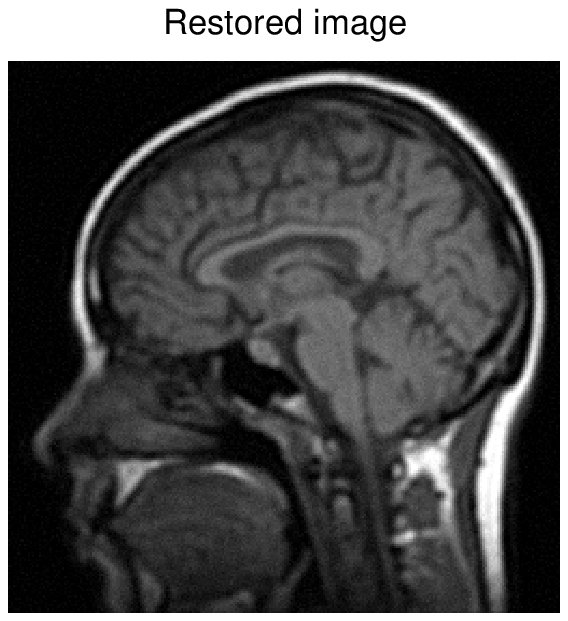} \newline
\includegraphics[width=0.32\textwidth]{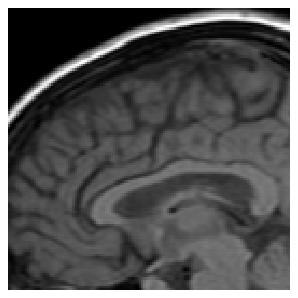} %
\includegraphics[width=0.32\textwidth]{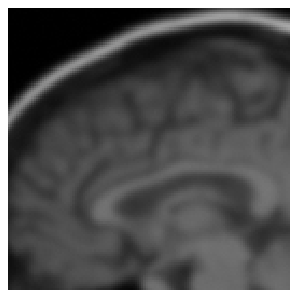} %
\includegraphics[width=0.32\textwidth]{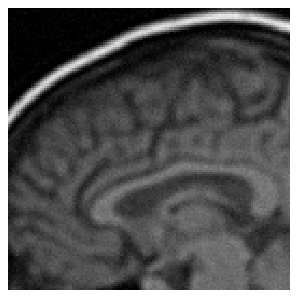} \newline
\includegraphics[width=0.32\textwidth]{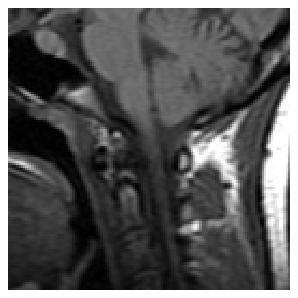} %
\includegraphics[width=0.32\textwidth]{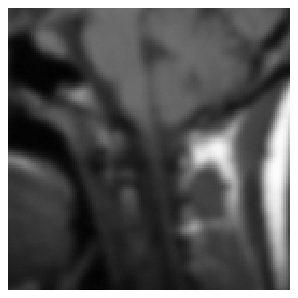} %
\includegraphics[width=0.32\textwidth]{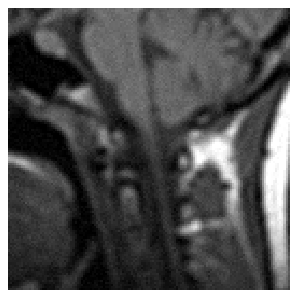}
\caption{\textit{Restoration of \texttt{mri.png} image. Original image,
blurred and noisy image with noise level $\protect\varepsilon =10^{-2}$ and
blur parameters \texttt{band}=7, \texttt{sigma}=2, restored image. From top
to bottom original size and two zoom.}}
\label{F6}
\end{figure}
The experiment has been carried out using Matlab 7.10 on a single processor
computer (Intel Core i7). The result has been obtained in 5 iterations of
the Arnoldi algorithm, in around $0.5$ seconds. Many other experiments on
image restoration have shown similar performances.

\section{Conclusion}

The fast convergence of the Arnoldi algorithm when applied to compact
operators makes the AT method particularly attractive for the regularization
of discrete ill-posed problems. The projected problems rapidly inherit the
basic features of the original one, allowing a substantial computational
advantage with respect to other approaches.

In this paper, in absence of information on the noise which affects the
right-hand side of the system, we have employed the GCV criterion. Contrary
to the hybrid techniques, the sequence of regularization parameters $\left\{
\lambda _{m}\right\} _{m\geq 1}$ is defined in order to regularize the
original problem instead of the projected one, leading to GCV approximations
which are similar to the ones used for pure iterative methods (\cite[\S 7.4]{PCH}).
Notwithstanding the intrinsic difficulties concerning the GCV
criterion, the arising algorithm has shown to be quite robust. Of course
there are cases in which the method fails, but the numerical experiments
presented are rather representative of what happens in general.

\end{document}